\documentclass[preprint,12pt,authoryear]{elsarticle}

\usepackage{setspace}
\usepackage{longtable}
\usepackage[utf8]{inputenc}
\usepackage[margin = 1in]{geometry}
\usepackage{amsmath, amssymb, amsthm, verbatim, amsfonts}
\usepackage{bigstrut}
\usepackage{mathtools}
\usepackage{pdflscape}
\usepackage{enumitem}
\usepackage{tabularx}
\usepackage{soul}
\usepackage{multirow}
\usepackage{float}
\usepackage{rotating} 
\usepackage{soul} 
\usepackage{graphicx}
\usepackage{subcaption}
\usepackage{url} 

\usepackage{algorithm}
\usepackage[noend]{algorithmic}

\usepackage{pgfplots} 
\usepackage{booktabs}	
\usetikzlibrary{patterns}
\usepackage{hyperref, doi}
\hypersetup{colorlinks = true, linkcolor = blue, citecolor = cyan}

\usepackage{pgfplotstable}

\usepackage{tikz}
\usetikzlibrary{shapes.geometric, arrows}

\tikzstyle{startstop} = [rectangle, rounded corners, minimum width=3cm, minimum height=1cm,text centered, draw=black, fill=gray!20]
\tikzstyle{process} = [rectangle, minimum width=3cm, minimum height=1cm, text centered, draw=black, fill=blue!10]
\tikzstyle{decision} = [diamond, minimum width=3cm, minimum height=1cm, text centered, draw=black, fill=red!10]
\tikzstyle{arrow} = [thick,->,>=stealth]

\allowdisplaybreaks

\usepackage{amsthm}
\newtheorem{thm}{Theorem}

\newtheorem{dfn}{Definition}

\newtheorem{problem}{Problem}

\journal{journal Computers \& Operations Research}

\makeatletter
\def\ps@pprintTitle{%
  \let\@oddhead\@empty
  \let\@evenhead\@empty
  \let\@oddfoot\@empty
  \let\@evenfoot\@oddfoot
}
\makeatother

\begin{document}

\begin{frontmatter}



\title{The Weighted Connected p-Median Problem}

	
\author[Yeditepe]{Murat Elhüseyni\corref{cor1}}
			\ead{murat.elhuseyni@yeditepe.edu.tr}

                \author[Sabanci]{Burak Kocuk}
			\ead{burakkocuk@sabanciuniv.edu}
			
			\author[InnoRenew,Szeged]{Miklós Krész}
			\ead{miklos.kresz@innorenew.eu}
			
			\address[Yeditepe]{%
				Yeditepe University, Industrial Engineering, 34755, Istanbul, Turkey}

                \address[Sabanci]{%
				Sabancı University, Industrial Engineering, 34956, Istanbul, Turkey}

                \address[InnoRenew]{%
				InnoRenew CoE, UP IAM \& UP FAMNIT, University of Primorska,  6000, Koper, Slovenia}

                \address[Szeged]{%
				University of Szeged, Department of Applied Informatics, 6725, Szeged, Hungary }

			\cortext[cor1]{Corresponding author}

\begin{abstract} 

The connected $p$-median problem is defined as a variant of the classical $p$-median problem when the 
facility nodes induce a connected subgraph. In this paper, we introduce the weighted version of the above problem when the weight of the facility connection in the objective function is defined by the minimum weight spanning tree of the facility nodes. This approach is motivated by the sink node selection in distributed sensor networks, in which the collected information is shared among the sink nodes through the minimum spanning tree. The weights of the graph determining the network topology of the candidate sink nodes as connection costs are distinguished from the standard access costs of the $p$-median problem. The fixed deployment costs for the setup of facilities are also  considered. The objective is to minimize the overall cost as the sum of deployment cost, access cost and connection cost. 
We show that the problem is NP-hard and propose three mixed-integer linear programming (MILP) formulations adapted from the traveling salesperson problem literature. Since these formulations are poorly scalable with respect to network size, we develop a four-phase matheuristic method based on linear programming rounding. We conduct an extensive computational study to evaluate the performance of the MILP formulations and 22 variants of the matheuristic under different parameter settings. The results indicate that the MILP models perform effectively on small instances but struggle to solve medium- and large-scale instances within a two-hour time limit. In contrast, several matheuristic variants consistently produce high-quality solutions within minutes. Finally, we analyze the impact of network structure, size, density, and the parameter $p$ on solution quality, providing further insights for network design.
\end{abstract}


\end{frontmatter}

\section{Introduction}

One of the primary problems in discrete facility location theory is the $p$-median problem in which $p$ nodes need to be selected in a network to locate facilities in such a way that the overall sum of the access costs from the rest of the network to the facilities is minimized. Access cost from a given node $n$ to the facilities is generally defined as the length of the shortest path in the network from $n$ to the closest facility. For a classical summary of solution methods the reader can consult \cite{solution-p-median}.

The connected $p$-median problem requires the facility nodes to induce a connected subgraph. A related problem is determining the minimum connected dominating set  when a connected subgraph $D$ of minimum cardinality is searched such that each node of the graph is connected by an edge to $D$. Connected dominating sets are thoroughly studied (for an early review see \cite{guha1998approximation}), but research on connected $p$-medians remains limited to special graph classes (see e.g. \cite{bai2021connected}, \cite{chang2016connected}). The state-of-the-art with the $p$-center and $p$-centdian problems (convex combination of $p$-median and $p$-center) is analogous: motivated by sensor network design the problems were studied for special graphs only (\cite{Nguyen2022ALT}, \cite{yen2007p}).

For the weighted connected $p$-median problem, we assume that the (undirected) network of candidate nodes for facility locations is settled with nonnegative weights on the edges. These weights are depicted as connection costs and independent from the access costs used for the classical $p$-median problem. We can also consider the deployment cost for each candidate facility node induced by installation, which is fixed. Since the objective function is naturally the overall sum of the access, connection and deployment costs,  using the minimum spanning tree for this above goal is a natural choice from a theoretical point of view on the one hand, however, on the other hand, the motivation of the problem originates from the application area of distributed system based sensor networks summarized in the following.

In conventional wireless sensor network (WSN) deployments, sensing devices acquire environmental data and forward it to a centralized infrastructure, typically cloud-based, where storage and analysis are performed \citep{akyildiz2002wireless}. While this architecture simplifies system design, it introduces several critical limitations, including increased communication delay, excessive bandwidth consumption, and potential risks related to data privacy and system security. To overcome these drawbacks, recent research has shifted toward decentralized processing approaches, most notably edge computing. This paradigm relocates computational tasks closer to where the data is generated, enabling intermediate nodes—such as gateway or sink devices—to handle operations locally \citep{shi2016edge}. This architecture can be organized as a distributed system based network \citep{mrissa2022privacy} in which we distinguish two levels of communication: sensor level and gateway level.  Each sensor is characterized by a specific bandwidth demand, representing the traffic capacity requirement that must be routed to a sink through the shortest-path distance over an access network, thereby incurring an \textit{access cost} \citep{andrews1998access}. The gateway level of the network comprises gateways that store, collect and process data in a distributed manner \citep{mrissa2022privacy}. Since sink nodes aggregate all data from sensors and decompose them to route messages to their respective destinations, they require substantial switching capabilities, leading to the \textit{deployment cost} \citep{andrews1998access}. Concerning budget constraints,  the number of gateways is limited to  a predefined number of $p$.  \citep{andrews1998access, mrissa2022privacy}. All collected data are shared among all the sink nodes for the gateway subnetwork to operate as a distributed system \citep{mrissa2022privacy}, but the cost of this data transmission (\textit{connection cost})  is different from the access cost and independent from the amount of data.  Also, while the sensor level communication is a single path transmission toward the nearest sink node, on the gateway level data are transmitted from each sink node on multiple paths. In the gateway deployment problem the goal is to identify the gateway nodes in a sensor network in such a way to minimize the overall cost of sensor level and gateway level communication as the sum of access costs, deployment costs and connection costs.

Concerning the weighted connected $p$-median problem, it is easy to see the analogous characteristics with the distributed system based sensor networks. In the objective function, the connection cost needs further explanation. When data are collected at a sink node, they are shared with the remaining gateways through multiple paths. To minimize the connection cost, this multi-routing structure is naturally represented by a minimum spanning tree. Therefore, the connection cost does depend on the gateway subnetwork only.  On the other hand, it needs to be scaled up by the number of packets, which is fixed either by the size of the network or by the number of sink nodes (depending on the technology). Consequently, by appropriate scaling of the edge weights, connection cost is realistically determined by minimum spanning tree. By the above analogy in the rest of the paper we will use the term ``sink nodes'' for the gateway nodes.

In summary, our goal in this paper can be described as follows:
Given a network of candidate sink nodes with sink deployment costs and edges weighted by connection costs, together with access costs from demand nodes to sink nodes, our objective is to select $p$ sink nodes, determine a spanning-tree backbone among them, and route demands to sinks in the most cost-efficient manner. 

Note that the motivating example as well as the test cases in the paper will also assume that the problem is metric with respect to access costs defined by the shortest paths in the network of demand and sink nodes. Nevertheless, our methodology does not request the metric assumption, it can be considered in the most general sense.

\subsection{Literature Review} 

When examining the literature on facility location in networks, existing studies can be broadly categorized into three main domains: (i) deployment cost, (ii) type of connection of facilities, and (iii) access types of sink nodes.

In the domain of deployment cost, one of the earliest contributions is presented by \citet{cornuejols1983uncapicitated}, who introduce the \textit{uncapacitated facility location problem}, where facilities incur deployment costs. In the context of sensor networks, \citet{andrews1998access} consider switching costs of sinks as deployment costs. In the same year, \citet{guha1998approximation} incorporate node weights into the weighted connected dominating set problem. Afterwards, \citet{swamy2004primal} extend the \textit{uncapacitated facility location problem} by requiring facilities to be connected. Note that the latter two studies ensure connectivity through a Steiner tree constructed over the selected medians; therefore, the resulting connectivity is weak. In this work, we also consider sink deployment costs.

The second category, namely the connection type of facilities, can be divided into two subcategories: weak connectivity and strong connectivity. In the weak connectivity domain, \citet{guha1998approximation} provide one of the earliest studies through the weighted connected dominating set problem. Later, \citet{karget2000building} address the maybecast problem, whose approximation algorithm is further refined by \citet{gupta2001provisioning} in the context of virtual private networks (VPNs). Afterwards, \citet{swamy2004primal} study the connected and $p$-connected facility location problems. In these studies, connectivity is established through Steiner-tree-based structures over the selected facilities. In this work, we strengthen these weak connectivity structures by enforcing strong connectivity through a spanning tree among the facilities.
Regarding strong connectivity, \citet{ravi1999approximation} appear to be among the first to investigate ring structures (closed loops) over sinks in telecommunication networks. Subsequently, \citet{yen2007p} address the connected $p$-center problem. \citet{ren2011note} study the minimum connected set cover problem, followed by \citet{elbassioni2012relation}, who consider its edge-weighted version. Afterwards, \citet{solmaz2014communication} investigate the connected $p$-center problem, although connectivity is achieved through wireless communication, resulting in a secondary connectivity graph. For special graph classes, \citet{chang2016connected} address the connected $p$-median problem on block graphs, whereas \citet{bai2021connected} solve the connected median problem on cactus graphs. Note that all of these connectivity studies, except \citet{ravi1999approximation}, also satisfy acyclicity because they construct tree structures, whereas \citet{ravi1999approximation} employ a loop structure. Unlike the existing literature, we enforce strong connectivity explicitly through a spanning-tree backbone.

Concerning the access type, the literature can be grouped into two subcategories:   studies with exactly $p$ facilities and  studies without any restrictions on the facility number. The first subcategory can further be divided into the $p$-median and $p$-center problems. In the $p$-median literature, early studies such as \citet{lin1992approximations}, \citet{lin1992geoapproximation}, and \citet{charikar1999constant} do not consider connectivity requirements. \citet{swamy2004primal} appear to be the first to incorporate connectivity into the $p$-median problem, followed by subsequent studies \citep{chang2016connected,bai2021connected,mrissa2022privacy}. In contrast, only a limited number of studies address the connected $p$-center problem, namely \citet{yen2007p} and \citet{solmaz2014communication}.
In the second subcategory, the number of facilities is unrestricted. Early works such as \citet{cornuejols1983uncapicitated} and \citet{andrews1998access} do not impose connectivity requirements among facilities. Among the remaining studies, some consider weak connectivity \citep{guha1998approximation,karget2000building,gupta2001provisioning,swamy2004primal}, whereas others address strong connectivity \citep{ravi1999approximation,ren2011note,elbassioni2012relation}. Our problem is most closely related to the connected $p$-median problem.

Table \ref{tab:my-table} provides an overview of location studies focusing on connectivity over graphs in chronological order. In terms of the main problem components, only a limited number of studies incorporate deployment costs into their formulations. Regarding connectivity, only a small subset of the literature neglects connectivity requirements altogether. In terms of access structure, studies that disregard deployment costs generally design networks with exactly $p$ facilities. The literature also differs with respect to several secondary characteristics. Regarding graph structure, it is evident that researchers have primarily focused on general graphs. The table further shows that the literature addresses both practical telecommunication network design problems and theoretically motivated combinatorial optimization problems. Notably, heuristic and predominantly approximation algorithms are widely employed to solve these problems whereas mixed-integer linear programming (MILP) models are used relatively less frequently.

In this study, we address a research gap by simultaneously considering deployment costs, strong acyclic connectivity through a spanning tree, and the weighted $p$-median problem. In contrast to the existing literature, we develop exact MILP formulations and a linear programming (LP) rounding-based matheuristic approach.

\begin{table}[H]
	\centering
	\caption{Selected studies on facility location in networks. } 
	\resizebox{\textwidth}{!}{%
		\begin{tabular}{|c|c|ccc|ccc|c|c|c|}
			\hline
			\textbf{Study} & \textbf{Deployment cost} & \multicolumn{3}{c|}{\textbf{Connection type}} & \multicolumn{3}{c|}{\textbf{Access type}} & \textbf{Graph type} & \textbf{Application area} & \textbf{Solution method} \\ \hline
			& \textbf{yes}        & \textbf{weak} & \textbf{strong} & \textbf{acyclicity} & \textbf{p-median} & \textbf{p-center} & \textbf{no restriction} & & & \\ \hline
			\citet{cornuejols1983uncapicitated} & \checkmark &  &  &  &  &  & \checkmark & general & logistics & MILP \\  
            \cite{lin1992geoapproximation} & & & & & \checkmark &  & & geometric & combinatorial & approximation alg \\
            \cite{lin1992approximations} & & & & & \checkmark &  & & general & combinatorial & approximation alg \\
			 \citet{andrews1998access} & \checkmark &  &  &  &  &  & \checkmark & general & telecommunication & approximation alg \\  
			\citet{guha1998approximation} & \checkmark & \checkmark &  & \checkmark &  &  & \checkmark & general & combinatorial & approximation alg \\  
				\citet{ravi1999approximation} &  &  & \checkmark &  &  &  & \checkmark & general & LAN & approximation alg \\    
            \cite{charikar1999constant} & & & & & \checkmark &  & & geometric & combinatorial & approximation alg \\ 
			 \citet{karget2000building} &  & \checkmark &  & \checkmark &  &  & \checkmark & general & telecommunication & approximation alg \\  
			\citet{gupta2001provisioning} &  & \checkmark &  & \checkmark &  &  & \checkmark & general & VPN & approximation alg \\  
			\citet{swamy2004primal} & \checkmark & \checkmark &  & \checkmark & \checkmark &  & \checkmark & general & logistics & approximation alg \\  
			 \citet{yen2007p} &  &  & \checkmark & \checkmark &  & \checkmark &  & tree & internet network & heuristic alg \\  
        \citet{ren2011note} & &  & \checkmark & \checkmark & & & \checkmark & general & combinatorial & approximation alg \\
        \citet{elbassioni2012relation} & &  & \checkmark & \checkmark & & & \checkmark & general & combinatorial & approximation alg \\
			 \citet{solmaz2014communication} &  &  & \checkmark & \checkmark &  & \checkmark &  & general & WSN & exact alg \\  
			 \citet{chang2016connected} &  &  & \checkmark & \checkmark & \checkmark &  &  & block & combinatorial & heuristic alg \\  
			 \citet{bai2021connected} &  &  & \checkmark & \checkmark & \checkmark &  &  & cactus & combinatorial & heuristic alg \\  
		 \citet{mrissa2022privacy} &  &  &  &  & \checkmark &  &  & general & WSN & MILP \\  
			This study & \checkmark   &  & \checkmark & \checkmark & \checkmark &  &  & general & WSN & MILP+matheuristic \\ \hline
			\multicolumn{11}{|l|}{\footnotesize alg: algorithm, MILP: Mixed Integer Linear Programming, LAN: local access network, WSN: wireless sensor network, VPN: Virtual private network} \\ \hline
		\end{tabular}%
	}
	\label{tab:my-table}
\end{table}

\subsection{Our Approach and Contributions} 

Considering that the above-mentioned challenges occur in both theoretical and practical settings, we formulate the following problem. Each node has a nonnegative demand (e.g., the expected amount of environmental data collected by a sensor node). The set of candidate sink nodes induces a connected network, where each potential sink incurs a deployment cost. Routing demand to any sink generates a access cost proportional to the demand, and links between selected sinks incur additional connection costs. The objective is to select the locations  for $p$ sinks forming a spanning tree, assign each demand node to one of the chosen sinks in such a way to minimize the total deployment, connection, and access costs.


To solve this problem, we propose a \textit{baseline} binary integer programming (BIP) formulation that accurately models all problem components except for sink connectivity. Connectivity is ensured through constraints derived from three approaches: Dantzig-Fulkerson-Johnson (DFJ)~\citep{dantzig1954solution}, Miller-Tucker-Zemlin (MTZ)~\citep{miller1960integer}, and a novel \textit{Flow} approach formulated in this research. The baseline model, combined with these connectivity approaches, results in three distinct mixed integer linear programming (MILP) models, where DFJ-based constraints are added in a lazy fashion.

To handle large instances, we design a four-phase matheuristic algorithmic framework with 22 versions, based on the linear programming (LP) rounding technique. A computational analysis is conducted to evaluate the performance of the MILP models and matheuristic variants on synthetically generated graphs commonly studied in the literature.

The key performance indicators (KPIs) are defined as solution quality and CPU time. We examine the impact of graph structure, the number of nodes, and the $p$ parameter on these indicators, as well as on the three cost components, across the MILP models. The performance of each matheuristic version is assessed based on its deviation from the best MILP objective value for the test instances. Finally, the matheuristic version that best balances the KPIs is tested on a large-scale case study, and its performance is evaluated by comparing the objective value of the best feasible solution to the LP relaxation value.                                                                                                                                                                                                  
The main contributions of this study are as follows:

\begin{itemize}
	\item To the best of our knowledge, this is the first study to address the weighted connected $p$-median problem and we establish the NP-hardness of the problem.
	\item We develop three mixed-integer linear programming (MILP) models, adapted from the traveling salesperson problem literature, to provide exact solutions for small to medium-sized instances.
	\item We propose a four-phase linear programming (LP) rounding-based matheuristic algorithmic framework, inspired by the works of \citep{ren2011note, lin1992approximations, lin1992geoapproximation, charikar1999constant}, to effectively solve large-scale instances.
\end{itemize}

The remainder of this work is organized as follows: Section~\ref{sec:problemFormulation} provides the details of the problem, its complexity and introduces the three MILP models. Section~\ref{sec:matheuristic} explains the four phase LP-rounding based matheuristic by giving the algorithm steps at each phase. Section~\ref{sec:numresults} accounts for the generation of test cases and presents the computational results. Section~\ref{sec:conclusions} draws conclusions from tests with   potential research directions.

Preliminary version of some of the presented results can be found in the extended conference abstract \citep{SOR}.

\section{Problem Formulation}
\label{sec:problemFormulation}

In this section we are giving the formal problem definition and the basic mathematical models. As we will see the problem is NP-hard, thus solution methodologies in the next section need to focus on heuristics. In our formulation we will assume that the problem is fully defined by a network structure, i.e. the candidate sink nodes forming a subset of the set of demand nodes mapped into a weighted network structure. In other words, a weighted network is induced by the demand nodes and this network provides the topological structure of the candidate sink node subnetwork as well. On the other hand, the connection cost assigned to an edge between candidate sink nodes is different from the ``basic access weight'' of the same edge. The access cost between a demand node and a candidate sink node is defined by the appropriate shortest path.
The above problem setup is realistic for most of the application areas (sensor networks e.g.) and it clearly defines a metric problem. Nevertheless, as previously highlighted, our methodological framework is not restricted on the metric cases, it can be applied for location problems where the demand nodes are not forming a network structure, but access costs are given directly as inputs between any demand node and sink node.

\subsection{Problem Setting} 



\label{sec:problemSetting}

Suppose that we are given an {undirected}  connected network denoted by $G=(N, E)$ with node set $N$ and edge set $E$. 
Without loss of generality, we can assume that nodes are labeled from $1$ to $|N|$, and every edge $(i,j)$ 
is bidirectional with a single weight of $w_{ij}$. 
We will denote the set of candidate sink locations as $F$, where $F \subseteq N$, out of which exactly $p$ should be deployed. In this network, we assume that each node routes demand $d_i$ to one of the deployed sinks. The \textit{access cost} for transmitting the data from node $i$ to the sink $j$ is $d_i \times t_{ij}$, where $t_{ij}$ is the shortest distance between the demand node $i$ and the sink $j$ calculated using the edge weights.
The \textit{deployment cost} of a sink at node $j$ is denoted as $f_j$. If two deployed sinks $i\in F$ and $j\in F$ are directly connected, then the \textit{connection cost} of $c_{ij}$ is incurred. Our aim in this problem is to find a spanning tree over the deployed $p$ sinks in such a way that  the sum of deployment, access and connection costs is minimized. Table~\ref{tab:problemnotation}, we provide the index sets and parameters for completeness. Since we work on an edge and node weighted graph, we refer to this problem as \textit{the weighted connected $p$-median problem}, which is formally defined as below:   
\begin{problem}[Weighted Connected p-median Problem (denoted by \texttt{WCpMP}]\label{prob:UpMSTP}
	Given the sets and parameters defined in Table~\ref{tab:problemnotation}, choose a set of sinks $F^* \subseteq F$ with $|F^*|=p$ and edges $E^* \subseteq E$  such that the induced graph $G^*=(F^*,E^*)$ forms a spanning tree and the total cost $\sum_{j\in F^*} f_j + \sum_{i\in N } d_i \min\{t_{ij} : j \in F^*\} + \sum_{(i,j)\in E^*} c_{ij}$ is minimized.
\end{problem}

\begin{table}[H]
	\small
	\caption{Sets and parameters.}
	\begin{center}
		\begin{tabular}{|p{2cm} p{13.5cm}|}
			\hline
			Sets & \\
			\hline
			$G$ & undirected network \\ 
			$N$ & set of demand locations (nodes) \\ 
			$F$ & set of candidate sink locations, $F \subseteq N$\\
			$E$ & set of edges \\
			$\delta_i$ & set of adjacent nodes of node $i$, $\delta_i= \{ j \vert (i,j) \in E \} \cup \{ j \vert (j,i) \in E \} $ \\ 
			\hline
			Parameters & \\
			\hline
			$p$ & the number of sinks to be deployed \\
			$d_i$ & demand at node $i \in N$  \\
			$f_i$ & deployment cost of a sink at candidate location $i \in F$ \\
			$w_{ij}$ & edge weight between nodes $i$ and   $j$ \\
			$c_{ij}$ & cost of connecting sinks $i$ and $j$, $(i,j) \in E$\\
			$t_{ij}$ & the shortest distance between node $i\in N$ and candidate sink location $j\in F$ \\
			\hline
		\end{tabular}
	\end{center}
	\label{tab:problemnotation}
\end{table}

To better illustrate the problem setting, we provide a toy example  in Figure \ref{fig:toygraph}, where the underlying graph has six nodes and seven edges.
At each node $i$, the deployment cost $f_i$ is given in italic style at the top of the node whereas the demand $d_i$ is given at the bottom of the node in a dashed box.
At each {edge} $(i,j)$, a pair of cost values are given, where the first number refers to the {edge weight} $w_{ij}$ and is given in boldface, and the second number refers to the connection cost $c_{ij}$. 
All costs are represented by the sign \$.

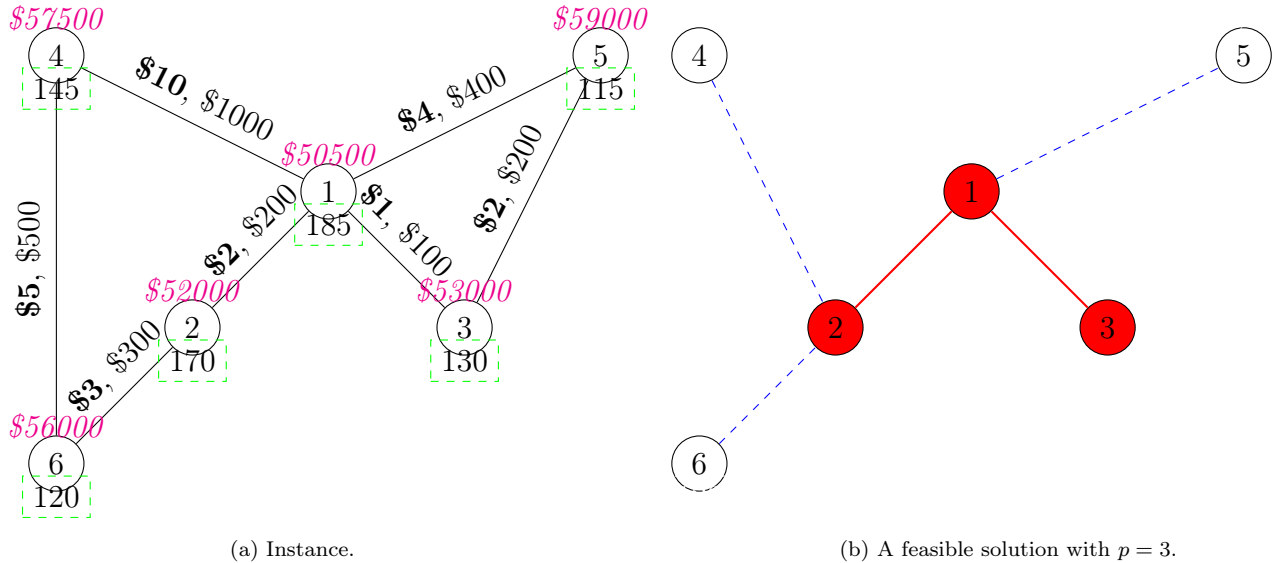
\begin{figure}[H]
	\caption{A toy example.}
	\label{fig:toyExample}
	\begin{subfigure}{0.475\textwidth}		
		\begin{tikzpicture}[scale=0.9]
			\node[draw, circle] (1) at (0,0) {6};
			\node[draw, circle] (2) at (2,2) {2};
			\node[draw, circle] (3) at (4,4) {1};
			\node[draw, circle] (4) at (6,2) {3};
			\node[draw, circle] (5) at (0,6) {4};
			\node[draw, circle] (6) at (8,6) {5};	
			
			\draw(1) edge node[pos=.6, above, sloped] { \textbf{\$3}, \$300} (2);
			\draw(1) edge node[midway, above, sloped] {\textbf{\$5}, \$500} (5);	
			\draw(5) edge node[midway, above, sloped] {\textbf{\$10}, \$1000} (3);
			\draw(2) edge node[pos=.6, above, sloped] {\textbf{\$2}, \$200} (3);
			\draw(3) edge node[pos=.4, above, sloped] { \textbf {\$1}, \$100} (4);
			\draw (3) edge node[midway, above, sloped] { \textbf{\$4}, \$400} (6);
			\draw (4) edge node[midway, above, sloped] {\textbf{\$2}, \$200} (6);	
			
			\node[above,magenta] at (2,2.2) {\textit{\$52000}};
			\node[above,magenta] at (4,4.2) {\textit{\$50500}};
			\node[above,magenta] at (0,6.2) {\textit{\$57500}};
			\node[above,magenta] at (0,0.2) {\textit{\$56000}};
			\node[above,magenta] at (6,2.2) {\textit{\$53000}};
			\node[above,magenta] at (8,6.2) {\textit{\$59000}};
			
			\node[above,draw=green,dashed] at (2,1.2) {170};			
			\node[above,draw=green,dashed] at (4,3.2) {185};
			\node[above,draw=green,dashed] at (0,5.2) {145};
			\node[above,draw=green,dashed] at (0,-0.8) {120};
			\node[above,draw=green,dashed] at (6,1.2) {130};
			\node[above,draw=green,dashed] at (8,5.2) {115};
		\end{tikzpicture}
		\caption{Instance.}
		\label{fig:toygraph}
	\end{subfigure}
	\begin{subfigure}{0.05\textwidth}		
		\hfill
	\end{subfigure}
	\begin{subfigure}{0.55\textwidth}		
		\begin{tikzpicture}[scale=0.9]
			\node[draw, circle] (1) at (0,0) {6};
			\node[draw, circle, fill=red] (2) at (2,2) {2};
			\node[draw, circle, fill=red] (3) at (4,4) {1};
			\node[draw, circle, fill=red] (4) at (6,2) {3};
			\node[draw, circle] (5) at (0,6) {4};
			\node[draw, circle] (6) at (8,6) {5};	
			
			
			\draw[red, thick] (2) -- (3);
			\draw[red, thick] (3) -- (4);
			\draw[blue, dashed] (3) -- (6);
			\draw[blue, dashed] (5) -- (2);
			\draw[blue, dashed] (1) -- (2);
			
			\node[above,white] at (8,6.2) {\textit{\$59,000}};
			\node[above,draw=white,dashed] at (0,-0.8) {\color{white}120};

			
		\end{tikzpicture}
		\caption{A feasible solution with $p=3$.}
		\label{fig:p=3sol}
	\end{subfigure}
\end{figure}

Figure \ref{fig:p=3sol} displays a solution for $p=3$, where $F^*=\{1,2,3\}$ and $E^*=\{(1,2),(1,3)\}$.
The dashed line shows  node-sink assignments. Note that the assignment between node 4 and sink 2 is performed through the shortest path over the network in Figure \ref{fig:toygraph}, {resulting in $t_{42}=w_{46}+w_{62}=\$8$}. The cost breakdown of this solution is as follows:

\begin{itemize}
	\item Deployment: $f_1+f_2+f_3 = 50500 + 52000 + 53000 = \$155500 $
	\item Connection: $c_{12} + c_{13} = 200 + 100 = \$300 $
	\item Access: 
	$ d_4 t_{42} + d_{5} t_{51} + d_6 t_{62}    = 145\times 8 + 115 \times 4 + 120 \times 3 =  \$ 1,980 $
\end{itemize}
Hence, the objective function value $Z$ of the toy example is 155500+300+1980 = \$157780.

\subsection{Problem Complexity}
In the following, we  formally prove the hardness of \texttt{WCpMP}.

\begin{thm} 
	\label{reduction}
	The weighted connected $p$-median problem is NP-Hard.
\end{thm}
\begin{proof} 
	Consider the decision version of the weighted connected $p$-median problem, which we denote by \texttt{WCpMP-D}: Does  there exist a feasible solution to Problem~\ref{prob:UpMSTP} with value at most~$\zeta^*$? We  prove that this problem is NP-Complete, from which the assertion of the theorem follows.
	
	We prove the NP-Completeness of \texttt{WCpMP-D} by a reduction from the decision version of the k-minimum spanning tree problem denoted by \texttt{kMSTP-D}, which is known to be NP-Complete \citep{ravi1996spanning}. 
	Consider an instance of \texttt{kMSTP-D}: Given a connected graph $G'=(N',E')$, edge weights $c_{ij}'$ for each $(i,j)\in E'$ and an integer $k$, does there exist  an induced tree $(\bar F, \bar E)$ with $|\bar F|=k$ such that $\sum_{ (i,j)\in \bar E} c_{ij} \le \zeta'$?
	We construct an instance of  \texttt{WCpMP-D}  as follows:
	\[
		N=F=N', \ E=E', \ p=k, \ d_i=f_i=0 \ i \in N', 
		c_{ij} = c_{ij}'\ (i,j)\in E', \ t_{ij}=0 \ i,j\in N', \ \zeta = \zeta' .
	\]
	Note that the size of the \texttt{WCpMP-D}  instance is polynomial in the size of the \texttt{kMSTP-D} instance.
	We now verify that \texttt{WCpMP-D} is feasible if and only if \texttt{kMSTP-D} is feasible.

	($\Rightarrow$) Consider a feasible solution of  \texttt{WCpMP-D} such that  $F^*\subseteq F$ with $|F^*|=p$ and $E^*\subseteq E$ such that the induced graph $(F^*,E^*)$ is a tree and  $\sum_{j\in F^*} f_j + \sum_{i\in N } d_i \min\{t_{ij} : j \in F^*\} + \sum_{(i,j)\in E^*} c_{ij} \le \zeta $. 
	We claim that a feasible solution to \texttt{WCpMP-D} is obtained by setting $\bar F=F^*$ and $\bar E = E^*$. In fact, by construction, we have $|\bar F|=|F^*|=p=k$, $(\bar F, \bar E)$ is a tree and $ \sum_{(i,j)\in \bar E} c_{ij}'  = \sum_{(i,j)\in \bar E^*} c_{ij}  \le \zeta = \zeta'$.

	($\Leftarrow$) Consider a feasible solution of  \texttt{kMSTP-D} such that $(\bar F, \bar E)$ with $|\bar F|=k$ such that $\sum_{ (i,j)\in \bar E} c_{ij} \le \zeta'$. 
	We claim that a feasible solution to \texttt{WCpMP-D} is obtained by setting $F^*=\bar F$ and $ E^*=\bar E$.  In fact, by construction, we have $|F^*|=|\bar F|=k=p$, $(F^*, E^*)$ is a tree and $  \sum_{j\in F^*} f_j + \sum_{i\in N } d_i \min\{t_{ij} : j \in F^*\} + \sum_{(i,j)\in E^*} c_{ij} =  \sum_{(i,j)\in \bar E} c_{ij}'    \le \zeta' = \zeta$. 
\end{proof}

\subsection{Mathematical Programming Models}
\label{sec:solmethods}

In this section, we will present three exact integer programming models of the weighted connected $p$-median problem. We tabulate the common binary decision variables in these models in Table~\ref{tab:commonDecVars}.

\begin{table}[htb]
	\small
	\caption{Common decision variables.}
	\begin{center}
		\begin{tabular}{ll}
			\hline
			$x_{ij}$ & 1 if node {\color{black}{$i \in N$}} is served by sink {\color{black}{$j \in F$}}; 0 otherwise \\
			$y_j$ & 1 if a sink is located at candidate site {\color{black}{$j \in F$}}; 0 otherwise\\
			$z_{ij}$ & 1 if a sink $i$ is connected to sink $j$, $(i,j) \in E, {\color{black}{i,j \in F}}$; 0 otherwise \\
			\hline
		\end{tabular}
	\end{center}
	\label{tab:commonDecVars}
\end{table}

Note that the deployment decisions $y_j$ and assignment decisions $x_{ij}$ are related to the $p$-median aspect of the problem, and the connection decisions $z_{ij}$ are needed to guarantee that the deployed facilities form a spanning tree. Some of the formulations to be described below utilize additional variables, which will be defined as needed.

\subsubsection{Baseline Model} 

We consider a variant of the $p$-median problem in which the selected facilities are required to form a spanning tree. This entails selecting exactly $p-1$ edges between $p$ deployed facilities, however, this alone does not guarantee the required spanning tree property.
Before presenting our exact models, we will first give a formulation that models all the aspects of the problem correctly except the connectedness of the facilities. This so-called \textit{Baseline Model} is the basis of all our exact models and given as follows:
\begin{subequations}   \label{eq:baselineModel}	
	\begin{align}
		\min &\hspace{0.5em}  \sum_{j \in F} f_j y_j + \sum_{i \in N} d_i \sum_{j \in F} t_{ij} x_{ij} + \sum_{(i,j) \in E, i,j \in F} c_{ij} z_{ij} \label{objfunc}\\
		\text{s.t.} &\hspace{0.5em}
		\sum_{j \in F} x_{ij} = 1  & i  \in N	\label{coverage}\\
		&\hspace{0.5em} x_{ij} \le y_j  & i \in N, j \in F \label{assigneligible}\\
		&\hspace{0.5em} \sum_{j \in F} y_j = p  & 	\label{p-facil}\\
		&\hspace{0.5em}  \sum_{(i,j) \in E, i,j \in F} z_{ij} = p-1  &  \label{numberofedges} \\
		&\hspace{0.5em} \sum_{j \in F} z_{ji} \le y_i  & i \in F \label{atmostone1} \\
		&\hspace{0.5em} z_{ij} +  z_{ji}  \le y_i & (i,j) \in E, i,j \in F, i<j \label{eq:selectoneedge-1}\\
		&\hspace{0.5em} z_{ij} +  z_{ji}  \le y_j & (i,j) \in E, i,j \in F, i<j	\label{eq:selectoneedge-2} \\
		&\hspace{0.5em} y_i \le \sum_{j \in \delta_i} z_{ij} + \sum_{j \in \delta_i} z_{ji} & i \in F \label{neighborconselect} \\
		&\hspace{0.5em} y_i \le \sum_{j \in \delta_i} y_{j} & i \in F	\label{neighbfacilselect} \\
		&\hspace{0.5em} x_{ij} \in \{0,1\} & i \in N, j \in F \label{transportvars} \\
		&\hspace{0.5em} y_{j} \in \{0,1\}  & j \in F \label{facilopen} \\
		&\hspace{0.5em} z_{ij} \in \{0,1\}  & (i,j) \in E, i \in F, j \in F  \label{arcvars} .
	\end{align}
\end{subequations}
The objective function \eqref{objfunc} minimizes the total sink deployment cost, access cost and connection cost. Constraint \eqref{coverage} states that each node is assigned to a single sink. Constraint \eqref{assigneligible} ensures that a node is served by a deployed sink. Constraint \eqref{p-facil} is used to make sure that exactly $p$ facilities are deployed. 
{\color{black}{Constraint~\eqref{numberofedges} stipulates the required number of edges in the tree. Constraint~\eqref{atmostone1} makes sure that given that a sink is deployed at a candidate node $i$, at most one incoming arc from a sink to the candidate node is allowed.}}
Given that there are facilities at candidate points $i$ and $j$, constraints \eqref{eq:selectoneedge-1}-\eqref{eq:selectoneedge-2} state that only one of the bilateral arcs between those nodes can be selected. 
{\color{black}{To strengthen the formula, we also add some valid inequalities.}} Utilizing the connection property of a spanning tree, constraint~\eqref{neighborconselect} accounts for the fact that if a sink is deployed, at least one of its adjacent edge must be selected to ensure  connection to another deployed sink. Based upon the same property, constraint~\eqref{neighbfacilselect} guarantees that if a sink is deployed, one of its neighbors must be selected as well. Lastly, variable domain restrictions are presented in constraints~\eqref{transportvars}-\eqref{arcvars}. 

Feasible solutions of problem~\eqref{eq:baselineModel} may not form a spanning tree among the deployed facilities since there might be cycles, or equivalently, the deployed facilities might not be all connected. The three exact formulations we will present below are different in how they handle this issue.

\subsubsection{Dantzig-Fulkerson-Johnson (DFJ) Based Model}
\label{sec:dfjModel}

A straightforward way to ensure the connectedness in formulation~\eqref{eq:baselineModel} is adopted from the Traveling Salesperson Problem (TSP) \citep{dantzig1954solution}. This approach entails the addition of constraint~\eqref{cyclebreaking}, which guarantees that  cycles of size at most $p-1$ are eliminated. 
\begin{equation}
	\sum_{(i,j) \in E, i,j \in W} z_{ij} \le \vert W \vert - 1 \quad W \subset F,  2 \le \vert W \vert < p \label{cyclebreaking}.
\end{equation}  
This way, we reach our first exact model  called the \textit{DFJ-Based Model} as \eqref{eq:baselineModel}, \eqref{cyclebreaking}. 

Note that the number of constraints in~\eqref{cyclebreaking} is exponential in $p$ and $|N|$. Therefore, in practice, these constraints are not added altogether but rather in a lazy fashion within the MILP solver. However, our preliminary experiments indicate that this model is not competitive against the other models introduced before (especially for larger values of $p$), therefore, it is not pursued further in this paper.

\subsubsection{Miller-Tucker-Zemlin  (MTZ) Based Model}
\label{sec:mtzModel}

In order to address the issue that ''constraint type'' \eqref{cyclebreaking} has exponentially many constraints, we formulate an alternative constraint \eqref{mtz1} adopted from {\color{black}{Hop Constrained Minimal Spanning Tree Problem (HMST) 
		\citep{gouveia1995using}}} which originally comes from \citet{miller1960integer} for subtour elimination as follows: 
\begin{align}
	&\hspace{0.5em} u_i - u_j + p z_{ij} \le p -1 & (i,j) \in E, i,j \in F	\label{mtz1} .
\end{align}
Here, $u_{i}$ is a new nonnegative  variable. The idea behind constraint \eqref{mtz1} is that when $z_{ij}=1$, it enforces $u_j \ge u_i + 1$, creating a strictly increasing order of $u$-values along selected arcs. This ordering eliminates subtours, as returning to a previously visited node would violate the monotonicity of the $u$-variables.


As a result, we obtain the second exact formulation called the \textit{MTZ-Based Model}  as \eqref{eq:baselineModel}, \eqref{mtz1}, which will be abbreviated as \texttt{MTZ}.

\subsubsection{Flow-Based Model}
\label{sec:flowModel}

Another way to replace exponentially many cycle-breaking constraints \eqref{cyclebreaking} is to ensure connectivity via a flow-based formulation. Suppose that an external supply of $p$ units will be sent to the network and the artificial data  of a sink $j$ is $y_j$, so that the external supply is delivered precisely to the deployed facilities. 
For this purpose, consider the following formulation:
\begin{subequations}\label{eq:flow based connected}
	\begin{align}
		&  s_j \le y_j & j& \in F \label{rootfacillink} \\
		&   \sum_{j\in F} s_{j} = 1 \label{eq:selectonetreeroot} & \\
		&   \sum_{j\to i} a_{ji} - \sum_{i\to j} a_{ij}   = y_i -  p s_i \quad & i& \in F \label{eq:artificial balance} \\
		& a_{ij} \le (p-1) z_{ij}  & i&,j \in F, (i,j)\in E \label{eq: a_zlink} \\
		&   a_{ij} \ge 0 \quad  &(&i,j)\in E \label{flowvar}  \\
		&  s_j \in \{0,1\}  & j& \in F \label{dummyvar} .
	\end{align}
\end{subequations}
Here, $s_j$ is a binary variable at node $j \in F$ representing whether it is the root of this artificial tree or not, and $a_{ij}$ are artificial flow variables between candidate facilities $i$ and $j$. {\color{black}{Note that unlike \citet{gollowitzer2011mip}, the root node is a decision variable here.}}
Constraint \eqref{rootfacillink} ensures that root of the flow is possible at one of the deployed facilities. Constraint \eqref{eq:selectonetreeroot} enforce that only one sink is selected as the real root of the flow. 
Constraint ~\eqref{eq:artificial balance} is related to flow balance in such a way that a candidate sink $i\in F$ has a unit data if it is deployed. Given that it is the source of flow, it sends $p$ units of flow to the network. Constraint \eqref{eq: a_zlink} represents a link between positive flows and an edge in the original graph. 

As a result, we obtain the third exact formulation called the \textit{Flow-Based Model}  as \eqref{eq:baselineModel}, \eqref{eq:flow based connected}, which will be abbreviated as \texttt{Flow}.

\section{Solution Methodology}
\label{sec:matheuristic}



We observe that solving the MILP models developed in Section~\ref{sec:solmethods} is particularly challenging, especially for large-scale instances. To address this, we propose a four-phase matheuristic algorithm that leverages LP rounding and reduces the sink-level graph to a more manageable size. The overall flow of the algorithm is illustrated in Figure~\ref{fig:flowsol}.

Given the problem parameters and one of the mathematical models introduced in Section~\ref{sec:solmethods}, we begin by solving its LP relaxation in Phase 1. In Phase 2, we filter out nodes and/or edges from the original graph whose associated decision variables have negligible values. Since this process may result in a disconnected graph, Phase 3 reconstructs a connected structure. However, the graph at this stage is only pseudo-feasible, as it may contain more than $p$ nodes or subcycles. Finally, in Phase 4, we restore feasibility by ensuring the correct number of sink nodes and eliminating any subcycles, followed by the assignment of demand nodes to sinks to obtain a feasible solution. 

\begin{figure}[h]
	\caption{Overview of the matheuristic algorithm.}
	\label{fig:flowsol}
	\begin{center}
		\begin{tikzpicture}[node distance=3cm]
			
			\tikzstyle{startstop} = [rectangle, rounded corners, minimum width=2cm, minimum height=1cm, text centered, draw=black, fill=gray!20]
			\tikzstyle{process} = [rectangle, minimum width=2cm, minimum height=1cm, text centered, draw=black, fill=blue!10]
			\tikzstyle{arrow} = [thick,->,>=stealth]
			
			\node (start) [startstop] {Model and Table 2};
			\node (phase1) [process, right of=start, xshift=2cm] {Phase 1: LP Solve};
			\node (phase2) [process, right of=phase1, xshift=2cm] {Phase 2: Filtering};
			\node (phase3) [process, below of=phase2, yshift=-0.05cm] {Phase 3: Pseudo-Feasible Solution};
			\node (phase4) [process, left of=phase3, xshift=-4cm] {Phase 4: Feasible Solution};
			\node (end) [startstop, left of=phase4, xshift=-2cm] {Terminate};
			
			\draw [arrow] (start) -- (phase1);
			\draw [arrow] (phase1) -- (phase2);
			\draw [arrow] (phase2) -- (phase3);
			\draw [arrow] (phase3) -- (phase4);
			\draw [arrow] (phase4) -- (end);
			
		\end{tikzpicture}
	\end{center}
\end{figure}
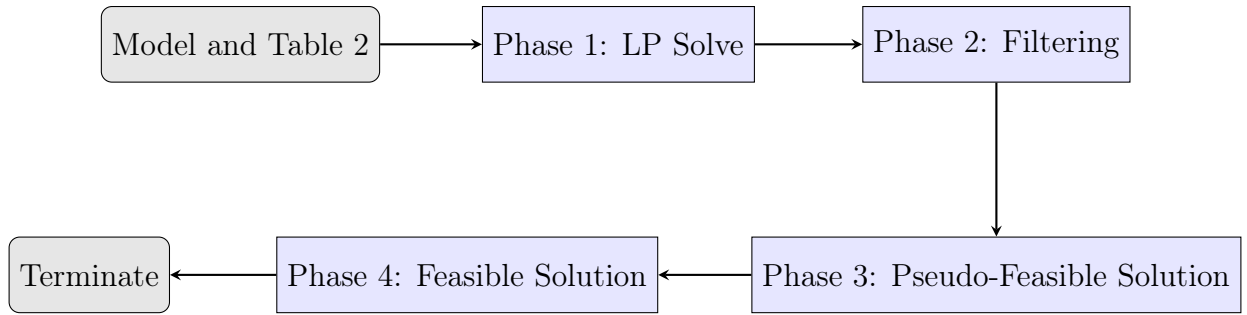

The overall algorithm has several parameters (e.g., \texttt{model}, \texttt{edgeFilter}, \texttt{nodeFilter}, $\gamma$, $\epsilon$, $\tau$, \texttt{reduction}), which will be introduced as needed. In our algorithms, {whenever there is a tie, we select the element with the smallest index}. Finally, {we use the notation $[i]$ to denote the elements of a set after a certain ordering.}

\subsection{Phase 1:  LP Solve} 

In this phase, we solve the LP relaxation of one of the  mathematical model introduced in Section~\ref{sec:solmethods}. We assume that $F = N$, as candidate sink locations are restricted to the set of demand locations (nodes) in all instances considered. The detailed steps of this phase is given in Algorithm~\ref{alg:solve-lp}. Here, the parameter \texttt{model} denotes the MILP model whose LP relaxation is chosen (either \texttt{MTZ} or \texttt{Flow}).

\begin{algorithm}[H]
	\caption{Phase 1:  LP Solve.}
	\label{alg:solve-lp}
	\begin{algorithmic}
		\REQUIRE \texttt{model} \hfill \COMMENT{Choose an LP relaxation.}
		\ENSURE  $\hat x_{ij}$, $\hat y_i$, $\hat z_{ij}$
		\IF{\texttt{model} is \texttt{MTZ}} 
		\STATE Solve the LP relaxation of Problem~\eqref{eq:baselineModel}, \eqref{mtz1}.
		\ENDIF
		\IF{\texttt{model} is \texttt{Flow}}
		\STATE Solve the LP relaxation of Problem~\eqref{eq:baselineModel}, \eqref{eq:flow based connected}.
		\ENDIF
	\end{algorithmic}
\end{algorithm}

\subsection{Phase 2: Filtering}

Given fractional values obtained from Phase 1, some of them are set to zero due to their low magnitudes. This process, referred to as ``filtering,'' aims to reduce the size of the original graph and obtain a candidate subgraph $G_c$.

Node filtering is controlled by the parameter \texttt{nodeFilter}, which can take one of the following values:
\begin{itemize}
	\item \texttt{LV}: \citet{lin1992approximations}'s filtering algorithm
	\item \texttt{LV-g}: \citet{lin1992geoapproximation}'s filtering algorithm for geometric medians
	\item \texttt{Ch}: \citet{charikar1999constant}'s filtering algorithm for geometric medians
	\item \texttt{False}: No node filtering is applied
\end{itemize}

Similarly, edge filtering is controlled by the parameter \texttt{edgeFilter}, which takes one of the following values:
\begin{itemize}
	\item \texttt{True}: Edge filtering is applied
	\item \texttt{False}: No edge filtering is applied
\end{itemize}

Following the filtering step, we construct candidate structures on the reduced graph. The main steps of this phase are presented in Algorithm~\ref{alg:filtering}. Briefly, Phase 2 constructs a reduced candidate subgraph $G_c=(F_c,E_c)$ from the fractional solution. Initially, we set candidate facilities $F_c$, candidate edges $E_c$, and all subsets to empty, and compute $\hat C_j$ for all $j \in N$, where $\hat C_j$ represents the weighted cost of assigning node $j$ to its medians in the fractional solution $\hat{x}$. Under \texttt{LV-g}, we further expand these subsets whenever they intersect. We then order the subsets by $\hat{C}_j$ and apply a greedy procedure to select nodes into $F_c$ so as to ensure coverage of $N$. If \texttt{nodeFilter} is \texttt{Ch}, we order nodes by $\hat{C}_j$, consolidate demands according to the specified conditions, and obtain a candidate set $N$ consisting of nodes with positive consolidated demand. If its size is larger than $p$,  we select a subset via sorting and a graph-based procedure to construct a dominating set $F_c$; otherwise, we designate the dominating set as the candidate set. In addition, we may apply edge filtering. In this case, we define $E_c$ as the set of edges in $E$ whose corresponding variables $\hat{z}_{ij}$ meet or exceed the threshold $\tau$. If no node filtering is applied, we induce $F_c$ from $E_c$; otherwise, we restrict $E_c$ to edges whose endpoints both lie in $F_c$. The algorithm outputs the candidate subgraph formed by $F_c$, $E_c$, and the subsets $S_j$.

\begin{algorithm}[]
	\caption{Phase 2:  Filtering.}
	\label{alg:filtering}
	\begin{algorithmic}
		\REQUIRE $\hat x_{ij}$, $\hat y_i$, $\hat z_{ij}$, \texttt{nodeFilter}, $\epsilon$, $\gamma$, \texttt{edgeFilter}, $\tau$ \hfill \COMMENT{Choose a filtering approach.}
		\ENSURE $F_c$, $E_c$, $\{S_j\}_{j \in N}$
		
		\STATE Set $F_c = \emptyset$,  $E_c = \emptyset$,
		$S_j = \emptyset$ for $j\in N$,
		%
		and $\hat C_j= \sum_{i \in N}  c_{ij} \hat x_{ij}$ for  $j \in N$.
		\IF{\texttt{nodeFilter} is \texttt{LV} or \texttt{LV-g}}   
		\STATE
		Set   $S_j = \{ i \in N: c_{ij} \le (1+\epsilon) \hat C_i \}$ for each $j\in N$ s.t. $\hat y_j > \gamma$.
		\IF{\texttt{nodeFilter} is \texttt{LV-g}} 
		\FOR{$i,j \in N$}
		\IF{ $i\neq j$ and $S_i \cap S_j \ne \emptyset$}  
		\STATE Set $S_i \leftarrow S_i \cup \{j\}$ and $S_j \leftarrow S_j \cup \{i\}$.   
		\ENDIF
		\ENDFOR
		\STATE Order $S_j$ sets in increasing order of $\hat C_j$.
		\STATE Set $N'=N$.
		\FOR{$i = 1,\dots,|N|$}
		\IF{$V' \cap S_{[i]} \neq \emptyset$}
		\STATE  $F_c \leftarrow F_c \cup \{[i]\}$. 
		\STATE $N' \leftarrow N' \setminus S_{[i]}$. 
		\ENDIF
		\ENDFOR
		\ENDIF
		\ENDIF

		\IF{\texttt{nodeFilter} is \texttt{Ch}}
		\STATE Order  candidate facilities in increasing order of $\hat C_j$.  
		\STATE Set $d'_i = d_i$ for $i\in N$.
		\FOR{$i = 1, \dots, |N|$}
		\FOR{$j = 1 ,\dots, i-1$}
		\IF{$d_{[j]}>0$ and $c_{[i][j]} \le 4 \hat C_{[j]}$}
		\STATE Set $d'_{[j]} = d'_{[i]} + d'_{[j]}$ and $d'_{[i]} = 0$.
		\ENDIF
		\ENDFOR
		\ENDFOR
		\STATE Set $N = \lbrace i \in N \vert d'_i>0\rbrace$.
		\IF{$\vert N  \vert>p$}
		\STATE Find $s_i = \arg\min \{ c_{ij} : j \in N_\}$ for each $i\in N$.
		\STATE Sort the indices in $N$ in decreasing order of value $d_i c_{i,s_i} - f_i$. 
		\STATE Let $N_1$ be the first $2p-\vert N  \vert$ indices in $N$ and $N_2 = N  \setminus N_1$.   
		\STATE Build a graph $H = (N_1 \cup N_2, \{ (i,s_i) : i\in N_2 \})$.
		\STATE Find a dominating set $F_c$ in $H$ such that $F_c \supseteq N_1$. 
		\ELSE
		\STATE $F_c = N$.
		\ENDIF
		
		\ENDIF

		\IF{\texttt{edgeFilter} is \texttt{True}}
		\STATE Set $E_c = \{ (i,j)\in E : \hat z_{ij} \ge \tau\}$.  
		
		\IF{\texttt{nodeFilter} is \texttt{False}}
		\STATE Set $F_c = \{ i \in N : \exists j \in \delta_i : (i,j) \in E_c \}$. 
		\ELSE
		\STATE Set $E_c  \leftarrow \{ (i,j)\in E_c : i \in F_c, j \in F_c \}$.
		\ENDIF
		
		\ENDIF
		
	\end{algorithmic}
\end{algorithm}

\subsection{Phase 3: Pseudo-feasible solution}

In this phase, we aim to construct a pseudo-feasible solution by ensuring the connectivity of the candidate subgraph and enforcing coverage conditions. We define a pseudo-feasible solution as follows:

\begin{dfn}[Pseudo-feasible solution]
	A solution that satisfies all constraints of the MTZ or Flow formulation but may violate constraint~\eqref{p-facil}.
\end{dfn}

To this end, we employ Algorithm~\ref{alg:pseudo-feasible-solution}. The algorithm begins by identifying the connected components of the candidate subgraph $G_c=(F_c,E_c)$. For each pair of components $(\ell,\ell')$, we solve a shortest path problem over the original graph $G$ with unit edge weights. Let $\alpha_{\ell,\ell'}$ denote the length of the shortest such path, and let $V_{\ell,\ell'}$ and $E_{\ell,\ell'}$ denote the corresponding node and edge sets. Using these values, we construct a complete graph on the components $\{1,\dots,L\}$ with edge weights $\alpha_{\ell,\ell'}$ and compute a minimum spanning tree $(F_t,E_t)$. We then update the candidate subgraph by augmenting $F_c$ and $E_c$ with the nodes and edges in this tree, thereby ensuring connectivity. If \texttt{reduction} is enabled and the covering structure is nonempty, we perform additional adjustments. First, for any node in candidate facilities with empty subset, we add this node to its subset.  Next, for any node not yet covered, we assign it to the subset of its closest facility. After obtaining a connected structure together with the associated subsets, we apply the connected set cover algorithm of \citet{ren2011note} to derive a pseudo-feasible solution.

Before presenting the details of the connected set cover algorithm, we introduce the following definitions from \citet{ren2011note}. 

\begin{dfn}
Let $N$ be a finite set and $\mathcal{S} = \{S_i \subseteq N : i = 1, \dots, n\}$ a collection of subsets. Let $G$ be a connected graph with node set $\mathcal{S}$. A connected set cover (CSC) $\mathcal{R}$ is a set cover of $N$ such that $\mathcal{R}$ induces a connected subgraph of $G$.
\end{dfn}

\begin{dfn}
Let $G=(N,E)$ and $\{S_j\}_{j \in F_c}$ be a collection of subsets. Let $Cov = \{S_j : j \in F_c\}$, and let $R \subseteq Cov$ with $R \neq \emptyset$. For any $S \in Cov \setminus R$, an $R$--$S$ path is an ordered sequence $P_S = \{S_0,S_1,\dots,S_k\}$ such that 
(i) $S_0 \in R$, 
(ii) $S_k = S$, and 
(iii) $S_1,\dots,S_k \in Cov \setminus R$.

We denote the consecutive pairs in $P_S$ by $E(P_S)$ and the set of newly covered nodes by $C(P_S)$, i.e., those covered by $P_S$ but not by $R$.
\end{dfn}

\begin{algorithm}[H]
	\caption{Phase 3:  Pseudo-feasible Solution.}
	\label{alg:pseudo-feasible-solution}
	\begin{algorithmic}
		
		\REQUIRE $F_c$, $E_c$, $\{S_j\}_{j \in N}$, \texttt{reduction} 
		\ENSURE A pseudo-feasible solution.
		
		\STATE   Find connected components of $G_c=(F_c, E_c)$ as $G_c^\ell = (F_c^\ell, E_c^\ell)$ for some $\ell=1,\dots,L$. 
		
		\FOR{$1 \le \ell < \ell' \le L$}
		
		\STATE Solve the Shortest Path Problem for each pair in the set $P_{\ell,\ell'} = \{ (i,j) \in F_c^l \times F_c^{l'} \}$ over the original graph $G$ with unit edge weights.
		
		\STATE Let $\alpha_{\ell,\ell'} $ be the length of a minimum shortest path over the set $P_{\ell,\ell'}$, and let
		$N_{\ell,\ell'}$ and $E_{\ell,\ell'}$  be the set of nodes and the set of edges
		in that path, respectively.
		
		\ENDFOR

		\STATE  Obtain a Minimum Spanning Tree $(F_t, E_t)$ over the complete graph with nodes $\{1,\dots,L\}$ and edges weights as $\alpha_{\ell,\ell'}$ for $1 \le \ell < \ell' \le L$.

		\STATE Set $F_c \leftarrow F_c \cup F_t$ and $E_c \leftarrow E_c \cup E_t$. 
		
		\IF{ \texttt{reduction} is True and $\cup_j S_j \neq \emptyset$}
		
		\WHILE {$\exists j\in F_c$ s.t. $S_j = \emptyset$}
		\STATE Set {$S_j \leftarrow S_j \cup \{j\}$}. 
		\ENDWHILE
		
		\WHILE {$ \exists i \in N$ s.t. $i \notin \cup_{j \in F_c} S_j$} 
		\STATE Find $j'= \arg\min \{ c_{ij} : j \in F_c \}$.
		\STATE Set {$S_{j'} \leftarrow S_{j'} \cup \{i\}$}. 
		\ENDWHILE

		\STATE Apply Ren and Zhao's connected set cover algorithm \citep{ren2011note} detailed in Algorithm~\ref{alg:ren-zhao}. 
		\ENDIF
		
	\end{algorithmic}
\end{algorithm}

We now describe the connected set cover in Algorithm \ref{alg:ren-zhao}. Let $J$ denote the selected nodes, $R$ the collection of selected subsets, and $U$ the set of covered nodes. Initially, we select $S_{j'}$ with maximum cardinality and set $J=\{j'\}$, $R=\{S_{j'}\}$, and $U=S_{j'}$. The set of uncovered nodes is represented by $N'$, and $Cov = \{S_j : j \in F_c\}$ denotes the family of all candidate subsets. At each iteration, we consider the remaining sets $R' = Cov \setminus R$. For each subset in the remaining sets that is adjacent to the selected nodes (either via graph adjacency or set intersection),  we compute a shortest $R$--$S_j$ path $P_{S_j}$. Among all such candidates, the algorithm selects the path $P_{S'}$ which minimizes the ratio of path length to the number of nodes covered along the path outside $R$ (newly covered nodes) We then incorporate the selected path $P_{S'}$ into the solution by adding edges in $E(P_{S'})$, updating $J$ with corresponding indices, expanding $R$, and updating $U$ with newly covered nodes. This process continues until all nodes are covered. Finally, we obtain the connected subgraph defined by $F_c = J$ and $E_c = E'_c$.

\begin{algorithm}[H]
	\caption{Ren and Zhao's Algorithm}
	\label{alg:ren-zhao}
	\begin{algorithmic}
		
		\REQUIRE $N$, $\{S_j\}_{j \in F_c}$, $E_c$
		\ENSURE $G_c=(E_c, F_c)$
		\STATE Choose $j'= \arg \max \{|S_j|: j \in F_c\}$.
		\STATE Set $J = \{j'\}, R=\{S_{j'}\}$, $U= S_{j'}$, $N'= N \setminus S_{j'}$, $E'_{c} = \emptyset$, $Cov = \{ S_j : j \in F_c\}$.
		\WHILE {{\color{black}$N' \setminus U \ne \emptyset$}}
		\STATE Set $R' \leftarrow Cov \setminus R$.
		\FOR {$S_j \in R'$}
		\IF {$\delta_j \cap J \ne \emptyset$ or $S_j \cap \{S_i : S_i \in R \} \ne \emptyset$}  
		\STATE Find a shortest $R-S_j$ path $P_{S_j}$. 
		\ENDIF
		\ENDFOR
		\STATE Select $P_{S'} = \arg \min \left\{ \dfrac{|P_S|}{|C(P_S)|} : S \in R', |C(P_S)|>0 \right\}$. 
		\FOR {$ (S_i,S_j) \in E(P_{S'})$}
		\STATE $E'_{c} \leftarrow E'_{c} \cup \{(i,j)\}$.
		\ENDFOR
		\STATE Set $R \leftarrow R \cup P_{S'}$. 
		\FOR {$ S_i \in P_{S'}$}
		\STATE Set $J \leftarrow J \cup \{i\}$.
		\ENDFOR
		\STATE Set $U \leftarrow U \cup C(P_{S'})$. 
		\ENDWHILE
		\STATE Set $F_c \leftarrow J$ and $E_c \leftarrow E'_{c}$.
		
	\end{algorithmic}
	
\end{algorithm}

\subsection{Phase 4: Feasible solution}

In the final phase, we employ Algorithm~\ref{alg:feasible-solution} to obtain a feasible solution. Initially, while the number of selected facilities is less than $p$, the algorithm iteratively augments $F_c$ by adding a node outside the current set that is closest to the selected facilities. The corresponding edges are recorded to update the edge set accordingly until the size of the facilities is equal to $p$. If exactly $p$ facilities are selected, we construct a minimum spanning tree over $G_c$ to ensure connectivity. Each node is then assigned to its closest facility, and the objective function value is computed. Otherwise, we solve a restricted MILP on the reduced subgraph $G_c$. Specifically, if \texttt{model} is \texttt{MTZ}, we solve the formulation given in \eqref{eq:baselineModel} and \eqref{mtz1}. If \texttt{model} is \texttt{Flow}, we instead solve the formulation in \eqref{eq:baselineModel} and \eqref{eq:flow based connected}. The resulting solution is feasible with respect to the original problem.


\begin{algorithm}[]
	\caption{Phase 4:  Feasible Solution.}
	\label{alg:feasible-solution}
	\begin{algorithmic}
		\REQUIRE \texttt{model}, $G_c =(F_c, E_c)$ 
		\ENSURE A feasible solution.
		\WHILE{ $|F_c| < p$ }
		\STATE Find $(i',j') = \arg\min\{ c_{ij} : i \in F_c, j \not\in F_c \}$. 
		\STATE Set $F_c \leftarrow F_c \cup \{ j'  \}$ and $E_c \leftarrow E_c \cup \{ (i',j') \}$. 
		\ENDWHILE
		
		\IF{ $|F_c| = p$ } 
		\STATE Find an MST $(F_c, E_c^*)$ over $G_c$.
		\STATE Assign each node $i \in N$ to the closest sink $j \in F_c$.
		\STATE Compute the objective function value. \\
		
		\ELSE
		\IF{\texttt{model} is \texttt{MTZ}} 
		\STATE Solve MILP~\eqref{eq:baselineModel}, \eqref{mtz1} with $F=F_c$ and $E=E_c$.
		\ENDIF
		\IF{\texttt{model} is \texttt{Flow}}
		\STATE Solve MILP~\eqref{eq:baselineModel}, \eqref{eq:flow based connected} with $F=F_c$ and $E=E_c$.
		\ENDIF
		
		\ENDIF
	\end{algorithmic}
\end{algorithm}

{\color{red}
	{

	}
}

\section{Numerical Results}
\label{sec:numresults}

\subsection{Test Instances}
\label{doe}

In order to test the effectiveness of our techniques, 
%
we use five instance families  with varying characteristics. 
For each instance family and for each number of nodes $n \in \{100, 200, 300, 400\}$, we generate five independent instances. 
	\begin{itemize}
		\item Erdős-Rényi (ER) Graphs: We generate random graphs following the
		Erdős-Rényi  procedure~\citep{erdos1960evolution}, where the probability of selecting an edge between two nodes is set as 0.30.
		\item Barabási-Albert (BA) Graphs: We generate random graphs following \citet{barabasi1999emergence} using the \texttt{networkx} package in Python~\citep{networkx}. These graphs model situations  where a small number of nodes have a large number of neighbors. We use {a preferential attachment parameter value of between 15 and 53}. 
		\item Benchmark (Bench) Graphs: We generate random graphs following \citet{lancichinetti2008benchmark} using the \texttt{networkx} package in Python~\citep{networkx}. These graphs  model communities in the network with different sizes. 
		We use the power law exponent for the degree distribution  value of 3,  the power law exponent for the community size distribution  value of 2 and the fraction of inter-community edges of 0.4. In addition, the desired average degree and the minimum size of communities are adjusted with respect to the number of nodes.
		\item Forest Fire (FF) Graphs:  We generate random graphs following \citep{leskovec2005graphs} using the \texttt{igraph} package in R~\citep{igraph}. 
		%
		%
		We use the forward burning probability  value of 0.5 and backward burning ratio of 0.8. 
		\item OR-Lib p-Median (pMed) Graphs: We use the  p-median graphs adopted from OR-LIB \citep{beasley1985note} by extending with appropriate connection costs.
	\end{itemize}
	We note that BA, Bench and FF Graphs are scale-free.
	%
    For a graph with $n \in \{100, 200, 300, 400\}$ nodes, the parameter $p$ is selected from $\{10, 20, \ldots, n/10\}$ in increments of 10. Using the parameters and procedures outlined above, we generate five samples for each $n$ value of each graph type.
	We report the number of nodes $|N|$, the average number of edges $|E|$ and the average density over five samples in   Table \ref{tab:graphtable}. 
	With respect to the density, we observe that the graph types are in increasing order for 
	pMed, Bench, BA, FF and ER.
	

	
	\setlength{\tabcolsep}{2pt}

	\begin{table}[H]
		\centering
		\footnotesize
		\caption{Statistics of instances with respect to graph types.}
		\label{tab:graphtable}
		\begin{tabular}{|c|ccc|}
			\hline
			{instance}    & $|N|$   & $|E|$       & {density} \\ \hline
			er-100    & 100 & 1477.60  & 0.30    \\
			er-200    & 200 & 5999.00  & 0.30    \\
			er-300    & 300 & 13417.20 & 0.30    \\
			er-400    & 400 & 23949.20 & 0.30    \\
			ba-100    & 100 & 1531.00  & 0.31    \\
			ba-200    & 200 & 4951.00  & 0.25    \\
			ba-300    & 300 & 10171.00 & 0.23    \\
			ba-400    & 400 & 17191.00 & 0.21    \\
			bench-100 & 100 & 870.40   & 0.17    \\
			bench-200 & 200 & 1815.20  & 0.09    \\
			bench-300 & 300 & 4727.40  & 0.11    \\
			bench-400 & 400 & 8570.20  & 0.11    \\
			ff-100    & 100 & 1169.00  & 0.23    \\
			ff-200    & 200 & 5487.00  & 0.27    \\
			ff-300    & 300 & 12611.60 & 0.28    \\
			ff-400    & 400 & 23351.60 & 0.29    \\
			pmed-100  & 100 & 200.00   & 0.04    \\
			pmed-200  & 200 & 800.00   & 0.04    \\
			pmed-300  & 300 & 1800.00  & 0.04    \\
			pmed-400  & 400 & 3200.00  & 0.04    \\ \hline
		\end{tabular}
	\end{table}

	

	

	
	For each edge, the  weight $w_{ij}$ is generated uniformly between 1 and 100 following \citep{beasley1985note}. Afterwards, an all-pairs shortest path algorithm is performed to find the unit access costs $t_{ij}$  from the edge weights $w_{ij}$. 
	{Motivated by the uncapacitated sink location literature \citep{ghosh2003neighborhood,pirkul1998multi,vasko2003large}, we select large deployment costs compared to the access costs}. 
In particular, $f_i$ is selected uniformly between 50000 and 60000 whereas $d_i$ is picked randomly between 100 and 200. 
Finally, {inspired by \citet{swamy2004primal}, the connection cost  $c_{ij}$ is selected $c_{ij}= M w_{ij}$ with $M=100$.}



We solve all instances in C++ using Visual Studio 2022 and CPLEX 22.1.1 solver. The experiments are conducted on a machine with two Intel(R) Xeon(R) Silver 4210R processors, 64 GB RAM, and 32 threads. We set the time limit as two hours for each MILP problem and keep all the other parameters at their default values. 






{\color{red}}

\subsection{Computational Results}

\subsubsection{Exact MILP Results}


In this section, we report the results of the three exact MILP models we formulate in Section~\ref{sec:solmethods} 
in Tables~\ref{tab:p=10}, \ref{tab:p=20}, \ref{tab:p=30} and \ref{tab:p=40} for instances with $p=10$, $p=20$, $p=30$ and $p=40$, respectively. 
We have three key performance indicators (KPI) to compare the MILP models:
\begin{itemize}
	\item Z: The objective function value of the incumbent solution reported by CPLEX upon termination.
	\item \% Gap: The relative optimality gap reported by CPLEX upon termination.
	\item Time: Time in seconds.
\end{itemize}
We note that each of these three indicators are averaged over five samples in the tables below. 


\setlength{\tabcolsep}{2pt}

According to the results reported in Table~\ref{tab:p=10} for the instances with $p=10$, we observe that the {DFJ}-based model implemented using a lazy constraint callback is significantly outperformed with respect to all of the three KPIs compared to   MTZ-based and Flow-based models. Therefore, it is not used for the instances with larger $p$ values. 
We observe that the MTZ-based model outperforms the Flow-based model in all three KPIs on the average. However, for the FF instances, the Flow-based model seems to be more successful.
\begin{table}[H]
	\centering
	\footnotesize
	\caption{MILP results for $p=10$}
	\label{tab:p=10}
	\begin{tabular}{|c|rrr|rrr|rrr|} 
		\hline
		& \multicolumn{3}{c|}{\texttt{DFJ}}         & \multicolumn{3}{c|}{\texttt{MTZ}}          & \multicolumn{3}{c|}{\texttt{Flow}}
		\\ \hline
		{instance} & {Z  \ \ \ \ \  } & {\% Gap} & {Time} & {Z  \ \ \ \ \  } & {\% Gap} & {Time} & {Z  \ \ \ \ \  } & {\% Gap} & {Time}  
		\\ \hline
		er-100    & 628851.80  & 0.01 & 20.40   & 628851.80  & 0.01 & 8.80    & 628851.80  & 0.00 & 11.60   \\
		er-200    & 678889.40  & 0.09 & 2189.40 & 678889.40  & 0.00 & 270.60  & 678889.40  & 0.01 & 649.00  \\
		er-300    & 712896.80  & 0.72 & 6439.20 & 712536.80  & 0.01 & 1293.60 & 715318.40  & 0.88 & 5388.20 \\
		er-400    & 766241.80  & 2.16 & 7239.80 & 760647.00  & 0.66 & 5718.60 & 765785.80  & 1.88 & 7205.40 \\
		ba-100    & 632582.00  & 0.01 & 14.60   & 632582.00  & 0.01 & 6.20    & 632582.00  & 0.00 & 8.00    \\
		ba-200    & 698518.00  & 0.01 & 1061.00 & 698524.00  & 0.01 & 111.40  & 698518.00  & 0.01 & 279.60  \\
		ba-300    & 748406.80  & 0.07 & 3638.60 & 748406.80  & 0.00 & 332.20  & 748406.80  & 0.01 & 987.80  \\
		ba-400    & 791238.60  & 0.60 & 7073.40 & 790514.60  & 0.01 & 1507.00 & 791042.40  & 0.38 & 4049.60 \\
		bench-100 & 706149.60  & 0.00 & 11.60   & 709856.20  & 0.01 & 12.60   & 709856.20  & 0.01 & 37.00   \\
		bench-200 & 957937.20  & 0.01 & 415.00  & 959731.40  & 0.01 & 112.80  & 959731.40  & 0.01 & 269.40  \\
		bench-300 & 934999.80  & 0.07 & 2206.80 & 938250.80  & 0.01 & 461.00  & 938250.80  & 0.01 & 1561.20 \\
		bench-400 & 977576.60  & 0.36 & 6204.20 & 978035.20  & 0.01 & 1008.40 & 978035.20  & 0.01 & 2281.80 \\
		ff-100    & 763265.00  & 0.01 & 140.40  & 763265.00  & 0.01 & 27.80   & 763269.60  & 0.00 & 5.20    \\
		ff-200    & 851811.40  & 1.27 & 6614.20 & 849727.60  & 0.74 & 4488.60 & 849603.60  & 0.00 & 202.20  \\
		ff-300    & 938499.80  & 3.12 & 6605.60 & 917638.00  & 0.74 & 4827.20 & 916277.40  & 0.45 & 2531.20 \\
		ff-400    & 994939.40  & 2.57 & 6186.80 & 979892.00  & 1.18 & 3524.80 & 974002.00  & 0.73 & 3703.40 \\
		pmed-100  & 1298246.00 & 0.01 & 504.80  & 1298246.00 & 0.01 & 45.00   & 1298246.00 & 0.01 & 49.20   \\
		pmed-200  & 1487942.00 & 1.86 & 6813.80 & 1485596.00 & 0.39 & 5468.40 & 1485596.00 & 0.34 & 4146.60 \\
		pmed-300  & 1619056.00 & 4.03 & 7222.40 & 1605230.00 & 0.99 & 5635.40 & 1604586.00 & 1.38 & 6182.80 \\
		pmed-400  & 1773792.00 & 5.23 & 7207.00 & 1735122.00 & 1.99 & 6680.80 & 1754268.00 & 3.61 & 7203.40 \\ \hline
		\textbf{Average} & 948092.00 &	1.11	& 3890.45	& \textbf{943577.13} &	\textbf{0.34}	& \textbf{2077.06}	& 944555.84 &	{0.49}& 2337.63\\
		\hline
	\end{tabular}
\end{table}

Our observations for the instances with $p=20,30,40$ reported in Tables~\ref{tab:p=20}-\ref{tab:p=40} are similar. In terms of primal solution quality and the relative optimality gap, the MTZ-based formulation is slightly better 
while the Flow-based formulation is about 13-19\% faster on average. 

\begin{table}[H]
	\centering
	\footnotesize
	\caption{MILP results for $p=20$}
	\label{tab:p=20}
	\begin{tabular}{|c|rrr|rrr|} 
		\hline		
		& \multicolumn{3}{c|}{\texttt{MTZ}}     & \multicolumn{3}{c|}{\texttt{Flow}}    \\ \hline
		{instance} & {Z  \ \ \ \ \  } & {\% Gap} & {Time} & {Z  \ \ \ \ \  } & {\% Gap} & {Time} \\	\hline
		er-200    & 1162234.00    & 0.01 & 260.20   & 1162234.00    & 0.01     & 628.40   \\
		er-300    & 1189236.00    & 0.01 & 1329.40  & 1189236.00    & 0.01     & 2328.20  \\
		er-400    & 1232326.00    & 0.33 & 6166.60  & 1232480.00    & 0.47     & 6788.20  \\
		ba-200    & 1185758.00    & 0.01 & 89.80    & 1185740.00    & 0.01     & 122.20   \\
		ba-300    & 1227608.00    & 0.01 & 1243.40  & 1227616.00    & 0.01     & 707.00     \\
		ba-400    & 1263790.00    & 0.01 & 2246.00    & 1263894.00    & 0.08     & 3713.40  \\
		bench-200 & 1415278.00    & 0.01 & 603.60   & 1415278.00    & 0.01     & 1661.00    \\
		bench-300 & 1395638.00    & 0.05 & 2494.80  & 1395772.00    & 0.07     & 3570.20  \\
		bench-400 & 1434308.00    & 0.10 & 3321.40  & 1435178.00    & 0.21     & 4030.40  \\
		ff-200    & 1336384.00    & 0.67 & 7270.40  & 1336390.00    & 0.01     & 157.20   \\
		ff-300    & 1394124.00    & 0.59 & 7256.40  & 1393802.00    & 0.01     & 1391.40  \\
		ff-400    & 1448354.00    & 0.37 & 4596.00    & 1448184.00    & 0.03     & 2745.00    \\
		pmed-200  & 1811936.00    & 0.60 & 5854.40  & 1813598.00    & 1.33     & 6926.00    \\
		pmed-300  & 1951690.00    & 1.87 & 7209.20  & 1954734.00    & 2.29     & 7206.00    \\
		pmed-400  & 2073218.00    & 3.00 & 7206.20  & 2077106.00    & 3.15 & 7204.60  \\ \hline
		\textbf{Average} &
		\textbf{1434792.13} &	\textbf{0.51}	& 3809.85	& 1435416.13 &	\textbf{0.51} &	\textbf{3278.61}\\
		\hline
	\end{tabular}
\end{table}


\begin{table}[H]
	\centering
	\footnotesize
	\caption{MILP results for $p=30$}
	\label{tab:p=30}
	\begin{tabular}{|c|rrr|rrr|} 
		\hline		
		& \multicolumn{3}{c|}{\texttt{MTZ}}     & \multicolumn{3}{c|}{\texttt{Flow}}    \\ \hline
		{instance} & {Z  \ \ \ \ \  } & {\% Gap} & {Time} & {Z  \ \ \ \ \  } & {\% Gap} & {Time} \\	\hline
		er-300    & 1682942.00 & 0.01 & 786.80  & 1682942.00 & 0.01 & 700.80  \\
		er-400    & 1720732.00 & 0.04 & 4428.60 & 1720958.00 & 0.11 & 5218.00 \\
		ba-300    & 1721452.00 & 0.01 & 638.60  & 1721478.00 & 0.01 & 588.60  \\
		ba-400    & 1756230.00 & 0.03 & 2229.00 & 1756178.00 & 0.01 & 2519.40 \\
		bench-300 & 1880554.00 & 0.01 & 2172.20 & 1880554.00 & 0.04 & 3622.60 \\
		bench-400 & 1915942.00 & 0.05 & 5068.40 & 1916352.00 & 0.22 & 5275.60 \\
		ff-300    & 1886732.00 & 0.38 & 7236.80 & 1886704.00 & 0.01 & 914.00  \\
		ff-400    & 1941240.00 & 0.29 & 4676.40 & 1940514.00 & 0.03 & 2856.20 \\
		pmed-300  & 2376864.00 & 1.39 & 7204.00 & 2378980.00 & 1.94 & 7204.00 \\
		pmed-400  & 2489740.00 & 2.65 & 7205.20 & 2530764.00 & 4.20 & 7204.20 \\ \hline
		\textbf{Average} &
		\textbf{1937242.80}	& \textbf{0.49}&	4164.60	& 1941542.40 &	0.66	& \textbf{3610.34}\\
		\hline
	\end{tabular}
\end{table}


\begin{table}[H]
	\centering
	\footnotesize
	\caption{MILP results for $p=40$}
	\label{tab:p=40}
	\begin{tabular}{|c|rrr|rrr|} 
		\hline		
		& \multicolumn{3}{c|}{\texttt{MTZ}}     & \multicolumn{3}{c|}{\texttt{Flow}}    \\ \hline
		{instance} & {Z  \ \ \ \ \  } & {\% Gap} & {Time} & {Z  \ \ \ \ \  } & {\% Gap} & {Time} \\	\hline
		er-400     & 2218522.00 & 0.02    & 4212.00    & 2218588.00 & 0.05    & 3747.40 \\
		ba-400     & 2255168.00 & 0.02    & 1928.60    & 2255190.00 & 0.01    & 2526.00 \\
		bench-400  & 2408006.00 & 0.03    & 5636.40    & 2408006.00 & 0.02    & 4240.00 \\
		ff-400     & 2439698.00 & 0.20    & 5842.20    & 2439692.00 & 0.01    & 2387.60 \\
		pmed-400   & 2937358.00 & 2.14    & 7208.20    & 2959232.00 & 2.88    & 7204.40 \\ \hline
		\textbf{Average} &
		\textbf{2451750.40} &	\textbf{0.48}	& 4965.48 &	2456141.60 &	0.59 &	\textbf{4021.08} \\
		\hline
	\end{tabular}
\end{table}

In all the experiments, we observe that the computation time increases with the number of nodes $|N|$, as expected. On the other hand, the effect of the parameter $p$ on the computation time is less clear as it might be easier or harder to solve the same instance with a larger value of $p$. 
The results also suggest that \% Gap increases with  the parameter $p$. 
Finally, we note that the pMed instances are the hardest because of their large \% Gap and Time, followed by the FF instances. The BA instances are  the easiest to solve. For $p\le 20$, solving the Bench instances is easier whereas models perform better under for the ER instances when $p \ge 30$. 



Recall that the objective function contains three cost components. Next, we analyze the cost breakdown with respect to these cost components.
Figure \ref{fig:costBreakdown} comprises four different charts with respect to different values of $p$ {for the solutions obtained from the MTZ-based model}. 
Based on this figure, we have the following observations:
\begin{itemize}
	\item For the fixed value of $p$, the access cost increases with the number of nodes $|N|$.
	\item As the parameter $p$ directly affects the number of sinks to be deployed and the number of edges to be connected, the connection and deployment costs  increase  with~$p$.
	\item  For the fixed value of $|N|$, the access cost decreases with $p$. This is due to the fact that having a larger number of sinks allows for each demand node to be connected to a close-by sink at a smaller access cost.
	%
	\item The total cost is  larger for the sparser graph types coming from
	Bench and pMed instances due to the increased access cost. 
	\item For the instances with $p=10$, the deployment cost has a significant contribution to the objective function, especially for ER and BA instances. For the Bench instances, deployment and access costs are similar  whereas the access cost constitutes the largest  portion for the  pMed instances. {This observation might be explained by the density of each graph type.}
	%
	%
	\item For the instances with larger $p$,  the deployment cost dominates the access cost. This is due to the fact that the deployment cost increases with $p$ while the access cost decreases $p$ (for fixed $|N|$) as explained above.
	\item For all the instances and the choice of $p$ parameter, the connection cost is the smallest component in the objective function.
\end{itemize}



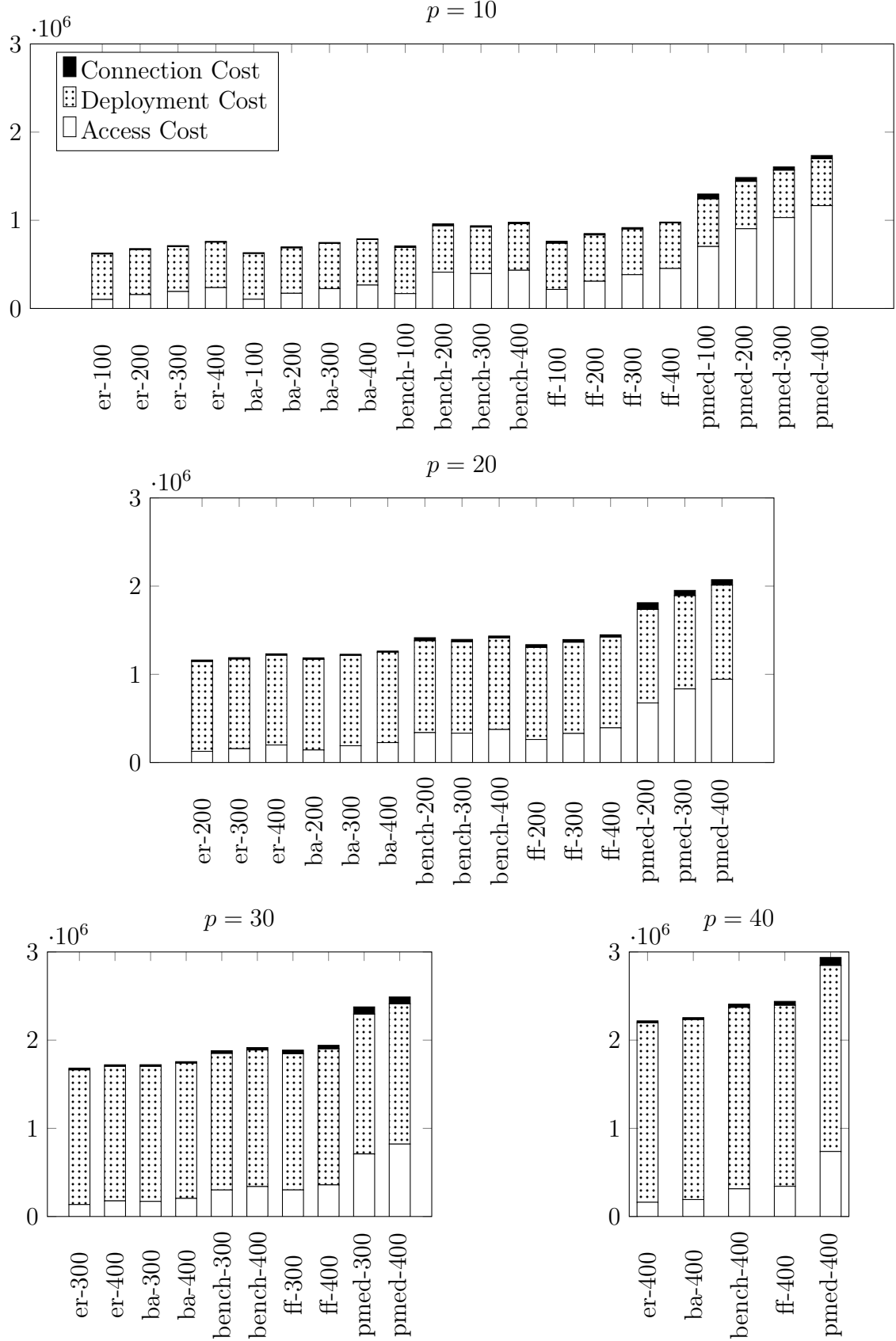
\begin{figure}
	\caption{Cost breakdown under the MTZ formulation.}
	\label{fig:costBreakdown}
	\pgfplotstableread{
		Label trans connect deploy
		er-100	102232.60	10620.00	515999.20
		er-200	157410.20	10220.00	511259.20
		er-300	192478.20	9320.00	510738.60
		er-400	236650.40	8300.00	515696.60
		ba-100	105011.60	8500.00	519070.40
		ba-200	172118.40	10480.00	515925.60
		ba-300	224800.60	9080.00	514526.20
		ba-400	265566.60	8480.00	516468.00
		bench-100	168732.20	17480.00	523644.00
		bench-200	411860.60	20580.00	527290.80
		bench-300	396362.00	14780.00	527108.80
		bench-400	433724.20	16540.00	527770.40
		ff-100	215469.00	22720.00	525076.00
		ff-200	308708.00	15860.00	525159.60
		ff-300	383424.40	16200.00	518013.60
		ff-400	453607.00	9000.00	517284.20
		pmed-100	702787.80	53920.00	541539.60
		pmed-200	904520.80	42840.00	538233.60
		pmed-300	1030927.40	36040.00	538260.60
		pmed-400	1166478.00	35140.00	533503.80
		
	}\testdata
	
	\begin{subfigure}{\textwidth}
		\centering
		\begin{tikzpicture}
			
			\begin{axis}[
				title={$p=10$},
				ybar stacked,
				ymin=0,
				ymax=3000000,
				xtick=data,
				width=16cm, height=6cm,
				legend style={cells={anchor=west}, legend pos=north west},
				reverse legend=true,
				xticklabels from table={\testdata}{Label},
				xticklabel style={text width=2cm,align=center,rotate=90},            
				]
				\addplot [fill=white] table [y=trans, meta=Label, x expr=\coordindex] {\testdata};
				\addlegendentry{Access Cost}
				
				\addplot [fill=gray!80, pattern=dots, point meta=y] table [y=deploy, meta=Label, x expr=\coordindex] {\testdata};
				\addlegendentry{Deployment Cost}
				
				\addplot [fill=black] table [y=connect, meta=Label, x expr=\coordindex] {\testdata};
				\addlegendentry{Connection Cost}
				
			\end{axis}
		\end{tikzpicture}
	\end{subfigure}

	\pgfplotstableread{
		Label trans connect deploy
		er-200	126524.20	16180.00	1019528.00
		er-300	156267.80	15120.00	1017848.00
		er-400	198460.40	15100.00	1018764.00
		ba-200	142038.60	16140.00	1027580.00
		ba-300	189854.40	13840.00	1023912.00
		ba-400	225545.00	13680.00	1024566.00
		bench-200	338675.00	33840.00	1042764.00
		bench-300	332609.40	25000.00	1038028.00
		bench-400	374280.20	21320.00	1038710.00
		ff-200	260956.20	31720.00	1043706.00
		ff-300	329518.40	29420.00	1035188.00
		ff-400	392980.40	25260.00	1030110.00
		pmed-200	675103.40	75460.00	1061372.00
		pmed-300	834317.00	58500.00	1058872.00
		pmed-400	943664.00	61320.00	1068234.00
	}\testdata
	
	\begin{subfigure}{\textwidth}
		\centering
		\begin{tikzpicture}
			\begin{axis}[
				title={$p=20$},
				ybar stacked,
				ymin=0,
				ymax=3000000,
				xtick=data,
				width=12cm, height=6cm,
				legend style={cells={anchor=west}, legend pos=north west},
				reverse legend=true,
				xticklabels from table={\testdata}{Label},
				xticklabel style={text width=2cm,align=center,rotate=90},
				]
				\addplot [fill=white] table [y=trans, meta=Label, x expr=\coordindex] {\testdata};
				
				\addplot [fill=gray!80, pattern=dots, point meta=y] table [y=deploy, meta=Label, x expr=\coordindex] {\testdata};
				
				\addplot [fill=black] table [y=connect, meta=Label, x expr=\coordindex] {\testdata};
				
			\end{axis}
		\end{tikzpicture}
	\end{subfigure}

	\pgfplotstableread{
		Label trans connect deploy
		er-300	135522.00	19800.00	1527622.00
		er-400	177076.00	17080.00	1526576.00
		ba-300	171195.80	18100.00	1532158.00
		ba-400	205422.20	15800.00	1535006.00
		bench-300	300865.00	30360.00	1549330.00
		bench-400	339802.80	27680.00	1548460.00
		ff-300	301481.60	40080.00	1545168.00
		ff-400	358543.60	38780.00	1543918.00
		pmed-300	709346.20	80500.00	1587018.00
		pmed-400	821695.20	76620.00	1591426.00
	}\testdata
	
	\begin{subfigure}{0.55\textwidth}
		\centering
		\begin{tikzpicture}
			
			\begin{axis}[
				title={$p=30$},
				ybar stacked,
				ymin=0,
				ymax=3000000,
				xtick=data,
				width=8cm, height=6cm,
				legend style={cells={anchor=west}, legend pos=north west},
				reverse legend=true,
				xticklabels from table={\testdata}{Label},
				xticklabel style={text width=2cm,align=center,rotate=90},
				]
				\addplot [fill=white] table [y=trans, meta=Label, x expr=\coordindex] {\testdata};
				
				\addplot [fill=gray!80, pattern=dots, point meta=y] table [y=deploy, meta=Label, x expr=\coordindex] {\testdata};
				
				\addplot [fill=black] table [y=connect, meta=Label, x expr=\coordindex] {\testdata};
				
			\end{axis}
		\end{tikzpicture}
	\end{subfigure}
	\pgfplotstableread{
		Label trans connect deploy
		er-400	161868.60	20520.00	2036134.00
		ba-400	191826.60	20720.00	2042620.00
		bench-400	313986.20	35460.00	2058560.00
		ff-400	342814.80	43180.00	2053704.00
		pmed-400	737305.00	89360.00	2110692.00
	}\testdata
	\begin{subfigure}{0.45\textwidth}
		\centering
		
		\begin{tikzpicture}
			
			\begin{axis}[
				title={$p=40$},
				ybar stacked,
				ymin=0,
				ymax=3000000,
				xtick=data,
				width=5.25cm, height=6cm,
				legend style={cells={anchor=west}, legend pos=north west},
				reverse legend=true,
				xticklabels from table={\testdata}{Label},
				xticklabel style={text width=2cm,align=center,rotate=90},
				legend style={legend columns=-1},
				]
				\addplot [fill=white] table [y=trans, meta=Label, x expr=\coordindex] {\testdata};
				
				\addplot [fill=gray!80, pattern=dots, point meta=y] table [y=deploy, meta=Label, x expr=\coordindex] {\testdata};
				
				\addplot [fill=black] table [y=connect, meta=Label, x expr=\coordindex] {\testdata};
				
			\end{axis}
		\end{tikzpicture}
	\end{subfigure}
	
\end{figure}

\newpage

\subsubsection{Heuristic Results}

In this section, we report the results of the heuristic approaches we develop in our paper. Based on the selections of parameters \texttt{model}, \texttt{nodeFilter}, \texttt{edgeFilter} and     \texttt{reduction} in Algorithms~\ref{alg:solve-lp}-\ref{alg:feasible-solution}, we have 22 different versions as summarized in Table~\ref{tab:param-setting-heuristic}. We note that we fix the constants as $\epsilon=1$, $\gamma=10^{-5}$ and $\tau=10^{-2}$ in our experiments.

\begin{table}[H]
	\centering
	\caption{Table of parameter settings for the heuristic approach. The markers are used in Figures~\ref{fig:deltaGap} and \ref{fig:lpGap}.}
	\label{tab:param-setting-heuristic}
	\begin{tabular}{|c|c|c|c|c|c|}
		\hline
		Version  &   \texttt{model} & \texttt{nodeFilter} & \texttt{edgeFilter} &     \texttt{reduction} & Marker \\
		\hline
		1&    \texttt{Flow} &      \texttt{False} &       \texttt{True} &      \texttt{False} & \multirow{2}{*}{\begin{tikzpicture}[scale=0.01]
				\begin{axis}   
					[xmin=-1,xmax=1,ymin = -1,ymax = 1,xtick={},ytick={}] 
					\addplot[mark size=500pt, mark=triangle*, blue] coordinates {(0,0)};
				\end{axis}
		\end{tikzpicture}}\\
		\cline{1-5}
		2&     \texttt{MTZ} &      \texttt{False} &       \texttt{True} &      \texttt{False} & \\
		\hline
		3&     \texttt{Flow} &         \texttt{LV} &       \texttt{True} &      \texttt{False} & \multirow{2}{*}{\begin{tikzpicture}[scale=0.01]
				\begin{axis}   
					[xmin=-1,xmax=1,ymin = -1,ymax = 1,xtick={},ytick={}] 
					\addplot[mark size=500pt, mark=halfdiamond*, mark color=black, fill=white] coordinates {(0,0)};
				\end{axis}
		\end{tikzpicture}}\\
		\cline{1-5}
		4&    \texttt{MTZ} &         \texttt{LV} &       \texttt{True} &      \texttt{False} & \\
		\hline
		5&    \texttt{Flow} &         \texttt{LV} &       \texttt{True} &       \texttt{True} & \multirow{2}{*}{\begin{tikzpicture}[scale=0.01]
				\begin{axis}   
					[xmin=-1,xmax=1,ymin = -1,ymax = 1,xtick={},ytick={}] 
					\addplot[mark size=500pt, mark=halfdiamond*, mark color=black, fill=black] coordinates {(0,0)};
				\end{axis}
		\end{tikzpicture}}\\
		\cline{1-5}
		6&    \texttt{MTZ} &         \texttt{LV} &       \texttt{True} &       \texttt{True} & \\
		\hline
		7&    \texttt{Flow} &         \texttt{LV} &      \texttt{False} &      \texttt{False} & \multirow{2}{*}{\begin{tikzpicture}[scale=0.01]
				\begin{axis}   
					[xmin=-1,xmax=1,ymin = -1,ymax = 1,xtick={},ytick={}] 
					\addplot[mark size=500pt, mark=halfdiamond*, mark color=white, fill=white] coordinates {(0,0)};
				\end{axis}
		\end{tikzpicture}}\\
		\cline{1-5}
		8&     \texttt{MTZ} &         \texttt{LV} &      \texttt{False} &      \texttt{False} & \\
		\hline
		9&    \texttt{Flow} &         \texttt{LV} &      \texttt{False} &       \texttt{True} & \multirow{2}{*}{\begin{tikzpicture}[scale=0.01]
				\begin{axis}   
					[xmin=-1,xmax=1,ymin = -1,ymax = 1,xtick={},ytick={}] 
					\addplot[mark size=500pt, mark=halfdiamond*, mark color=white, fill=black] coordinates {(0,0)};
				\end{axis}
		\end{tikzpicture}}\\
		\cline{1-5}
		10&     \texttt{MTZ} &         \texttt{LV} &      \texttt{False} &       \texttt{True} & \\
		\hline
		11&    \texttt{Flow} &       \texttt{LV-g} &       \texttt{True} &      \texttt{False} & \multirow{2}{*}{\begin{tikzpicture}[scale=0.01]
				\begin{axis}   
					[xmin=-1,xmax=1,ymin = -1,ymax = 1,xtick={},ytick={}] 
					\addplot[mark size=500pt, mark=halfcircle*, fill=red, mark color=white] coordinates {(0,0)};
				\end{axis}
		\end{tikzpicture}}\\
		\cline{1-5}
		12&     \texttt{MTZ} &       \texttt{LV-g} &       \texttt{True} &      \texttt{False} & \\
		\hline
		13&    \texttt{Flow} &       \texttt{LV-g} &       \texttt{True} &       \texttt{True} & \multirow{2}{*}{\begin{tikzpicture}[scale=0.01]
				\begin{axis}   
					[xmin=-1,xmax=1,ymin = -1,ymax = 1,xtick={},ytick={}] 
					\addplot[mark size=500pt, mark=halfcircle*, fill=red, mark color=red] coordinates {(0,0)};
				\end{axis}
		\end{tikzpicture}}\\
		\cline{1-5}
		14&     \texttt{MTZ} &       \texttt{LV-g} &       \texttt{True} &       \texttt{True} & \\
		\hline
		15&    \texttt{Flow} &       \texttt{LV-g} &      \texttt{False} &      \texttt{False} & \multirow{2}{*}{\begin{tikzpicture}[scale=0.01]
				\begin{axis}   
					[xmin=-1,xmax=1,ymin = -1,ymax = 1,xtick={},ytick={}] 
					\addplot[mark size=500pt, mark=halfcircle*, fill=white, mark color=white] coordinates {(0,0)};
				\end{axis}
		\end{tikzpicture}}\\
		\cline{1-5}
		16&     \texttt{MTZ} &       \texttt{LV-g} &      \texttt{False} &      \texttt{False} & \\
		\hline
		17&    \texttt{Flow} &       \texttt{LV-g} &      \texttt{False} &       \texttt{True} & \multirow{2}{*}{\begin{tikzpicture}[scale=0.01]
				\begin{axis}   
					[xmin=-1,xmax=1,ymin = -1,ymax = 1,xtick={},ytick={}] 
					\addplot[mark size=500pt, mark=halfcircle*, fill=white, mark color=red] coordinates {(0,0)};
				\end{axis}
		\end{tikzpicture}}\\
		\cline{1-5}
		18&     \texttt{MTZ} &       \texttt{LV-g} &      \texttt{False} &       \texttt{True} & \\
		\hline
		19&     \texttt{Flow} &         \texttt{Ch} &       \texttt{True} &      \texttt{False} & \multirow{2}{*}{\begin{tikzpicture}[scale=0.01]
				\begin{axis}   
					[xmin=-1,xmax=1,ymin = -1,ymax = 1,xtick={},ytick={}] 
					\addplot[mark size=500pt, mark=halfsquare*, mark color=green, fill=white] coordinates {(0,0)};
				\end{axis}
		\end{tikzpicture}}\\
		\cline{1-5}
		20&      \texttt{MTZ} &         \texttt{Ch} &       \texttt{True} &      \texttt{False} & \\
		\hline
		21&     \texttt{Flow} &         \texttt{Ch} &      \texttt{False} &      \texttt{False} & \multirow{2}{*}{\begin{tikzpicture}[scale=0.01]
				\begin{axis}   
					[xmin=-1,xmax=1,ymin = -1,ymax = 1,xtick={},ytick={}] 
					\addplot[mark size=500pt, mark=halfsquare*, mark color=white, fill=white] coordinates {(0,0)};
				\end{axis}
		\end{tikzpicture}}\\
		\cline{1-5}
		22&      \texttt{MTZ} &         \texttt{Ch} &      \texttt{False} &      \texttt{False} &  \\
		\hline
	\end{tabular}  
\end{table}

We have three KPIs to compare the heuristics:
\begin{itemize}
	\item Time: Time in seconds.
	\item \% Primal Gap: Computed as $100 \times \frac{z_{\text{Heur}}-z_{\text{MILP}}}{ z_{\text{MILP}}}$, where $z_{\text{Heur}}$ is the objective function value given by a heuristic obtained in Algorithm~\ref{alg:feasible-solution} and $z_{\text{MILP}}$ is the best objective function value given by the MILP models. 
	\item \% Dual Gap: Computed as $100 \times 
	\frac{z_{\text{Heur}}-z_{\text{LP}}}{ z_{\text{Heur}}}$, where $z_{\text{LP}}$ is the objective function value of the LP relaxation obtained in Algorithm~\ref{alg:solve-lp}. 
\end{itemize}
Notice that \% Primal Gap measures the relative difference of the objective value of the heuristic solution with respect to the best known solution obtained via any of the MILP models. This metric is only applicable if the MILP models are solved in advance. On the other hand,  \% Dual Gap can be computed even if the solutions of the  MILP models are not available.

\newcommand\labelcolor{white}

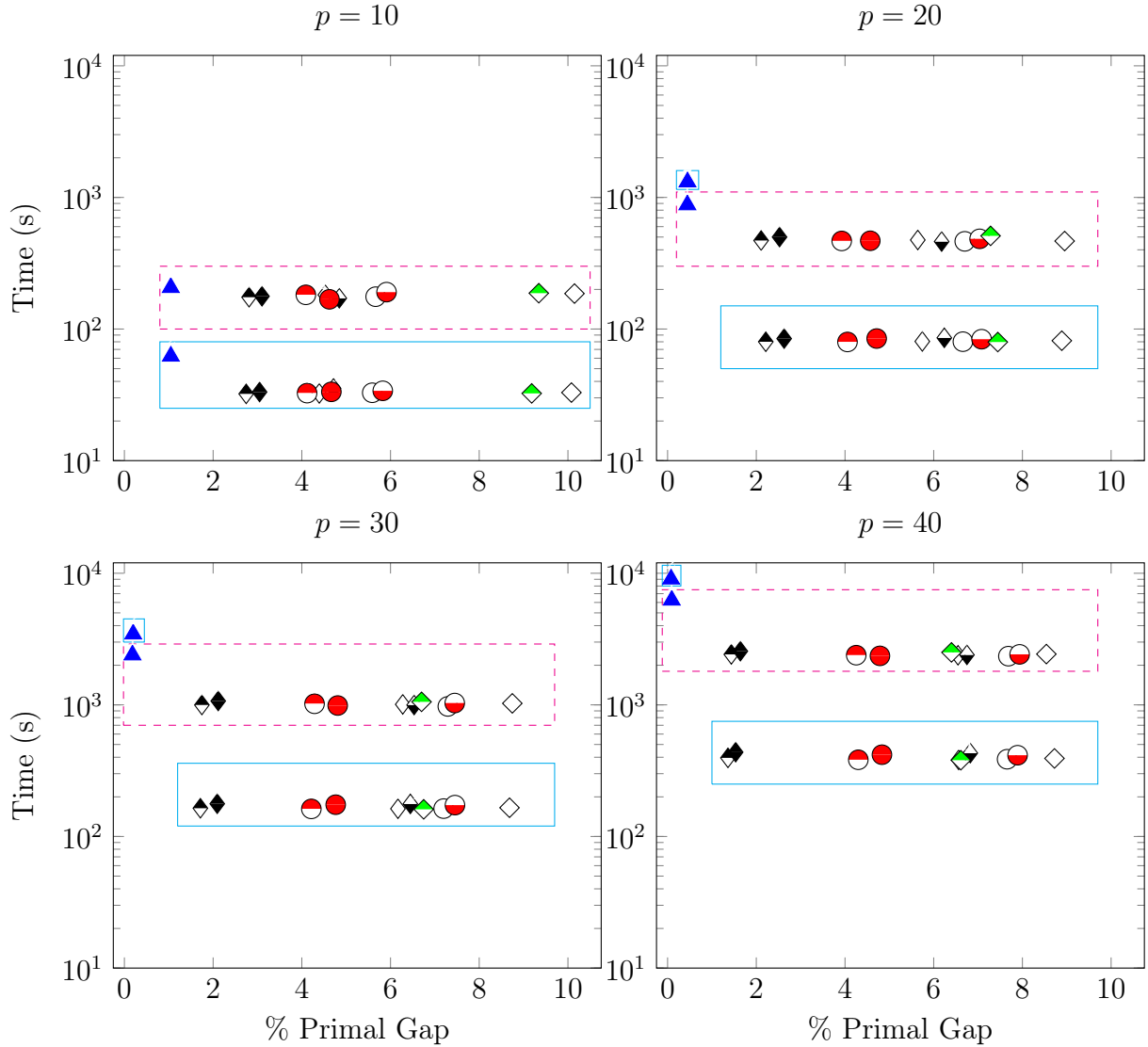
\begin{figure}[H]
	\caption{\% Primal Gap vs. Time. The markers in the rectangle with solid (resp. dashed) border use  the \texttt{MTZ} (resp. \texttt{Flow}) formulation.}
	\label{fig:deltaGap}
	
	\begin{subfigure}{.5\textwidth}
		\begin{tikzpicture} 
			\begin{axis}[
				title={{$p=10$}},
				xmin=-0.25, 
				xmax=10.75,
				ymin=10, 
				ymax=12000,
				ymode=log,
				ylabel={Time (s)},
				scatter/classes={
					t={scale=1.95,mark=triangle*, blue}, 
					d00={scale=1.95,mark=halfdiamond*, mark color=white, fill=white}, 
					d01={scale=1.95,mark=halfdiamond*, mark color=white, fill=black}, 
					d10={scale=1.95,mark=halfdiamond*, mark color=black, fill=white}, 
					d11={scale=1.95,mark=halfdiamond*, mark color=black, fill=black}, 
					c00={scale=1.95,mark=halfcircle*, fill=white, mark color=white}, 
					c01={scale=1.95,mark=halfcircle*, fill=white, mark color=red}, 
					c10={scale=1.95,mark=halfcircle*, fill=red, mark color=white}, 
					c11={scale=1.95,mark=halfcircle*, fill=red, mark color=red}, 
					s00={scale=1.95,mark=halfsquare*, mark color=white, fill=white},
					s10={scale=1.95,mark=halfsquare*, mark color=green, fill=white}
				},
				] 
				\addplot[
				scatter, 
				only marks,
				scatter src=explicit symbolic,
				nodes near coords*={\annotvalue},
				node near coord style={ anchor=\anchorvalue, font=\scriptsize, color=\labelcolor},
				visualization depends on={value \thisrow{annotation} \as \annotvalue },
				visualization depends on={  value \thisrow{anchor} \as \anchorvalue },
				]
				table[meta=label] {
					x       y            label  annotation  anchor
					1.046	206.50	t	1	south
					1.047	62.15	t	2   	south
					2.811	174.55	d10	3 south
					2.746	32.17	d10	4 south
					3.102	177.33	d11	5	south
					3.048	33.22	d11	6    south
					4.531	181.68	d00	7	south
					4.394	32.41	d00	8	south
					4.845	171.36	d01	9	south
					4.715	35.40	d01	10	south
					4.091	181.94	c10	11	south
					4.119	32.59	c10	12	south
					4.623	168.19	c11	13	south
					4.666	33.33	c11	14	south
					5.665	176.09	c00	15	south
					5.588	32.79	c00	16	south
					5.910	190.72	c01	17	south
					5.825	33.93	c01	18	south
					9.343	187.32	s10	19	south
					9.183	32.43	s10	20	south
					10.148	185.97	s00	21	south
					10.079	32.90	s00	22	south
				};
				
				\addplot[magenta, dashed] coordinates
				{(0.8,100) (10.5,100) (10.5,300) (0.8,300) (0.8,100)};
				
				\addplot[cyan] coordinates
				{(0.8,25) (10.5,25) (10.5,80) (0.8,80) (0.8,25)};
			\end{axis}
			
		\end{tikzpicture}
	\end{subfigure}
	\begin{subfigure}{.5\textwidth}
		\begin{tikzpicture}
			
			\begin{axis}[
				title={{$p=20$}},
				xmin=-0.25, 
				xmax=10.75,
				ymin=10, 
				ymax=12000,
				ymode=log,
				scatter/classes={
					t={scale=1.95,mark=triangle*, blue}, 
					d00={scale=1.95,mark=halfdiamond*, mark color=white, fill=white}, 
					d01={scale=1.95,mark=halfdiamond*, mark color=white, fill=black}, 
					d10={scale=1.95,mark=halfdiamond*, mark color=black, fill=white}, 
					d11={scale=1.95,mark=halfdiamond*, mark color=black, fill=black}, 
					c00={scale=1.95,mark=halfcircle*, fill=white, mark color=white}, 
					c01={scale=1.95,mark=halfcircle*, fill=white, mark color=red}, 
					c10={scale=1.95,mark=halfcircle*, fill=red, mark color=white}, 
					c11={scale=1.95,mark=halfcircle*, fill=red, mark color=red}, 
					s00={scale=1.95,mark=halfsquare*, mark color=white, fill=white},
					s10={scale=1.95,mark=halfsquare*, mark color=green, fill=white}
				},
				] 
				
				\addplot[
				scatter, 
				only marks,
				scatter src=explicit symbolic,
				nodes near coords*={\annotvalue},
				node near coord style={ anchor=\anchorvalue, font=\scriptsize, color=\labelcolor},
				visualization depends on={value \thisrow{annotation} \as \annotvalue },
				visualization depends on={  value \thisrow{anchor} \as \anchorvalue },
				]
				table[meta=label] {
					x       y            label  annotation  anchor
					0.446	877.55	t	1	south
					0.449	1310.03	t	2	south
					2.109	470.80	d10	3 north
					2.214	80.05	d10	4 north
					2.524	499.92	d11	5	south
					2.628	84.55	d11	6	south
					5.641	475.05	d00	7	south
					5.743	80.49	d00	8	south
					6.179	462.12	d01	9	south
					6.239	85.72	d01	10	south
					3.926	467.23	c10	11	south
					4.056	79.79	c10	12	south
					4.570	468.48	c11	13	south
					4.714	84.68	c11	14	south
					6.700	463.77	c00	15	south
					6.658	80.16	c00	16	south
					7.035	484.57	c01	17	south
					7.079	83.68	c01	18	south
					7.282	510.55	s10	19	south
					7.445	79.68	s10	20	south
					8.952	466.17	s00	21	south
					8.888	81.51	s00	22	south
				};

				\addplot[magenta, dashed] coordinates
				{(0.2,300) (9.7,300) (9.7,1100) (0.2,1100) (0.2,300)};
				
				
				\addplot[cyan] coordinates
				{(1.2,50) (9.7,50) (9.7,150) (1.2,150) (1.2,50)};
				
				
				\addplot[cyan] coordinates
				{(0.2,1150) (0.7,1150) (0.7,1600) (0.2,1600) (0.2,1150)};
				
			\end{axis}
			
		\end{tikzpicture}
	\end{subfigure}
	
	\begin{subfigure}{.5\textwidth}
		\begin{tikzpicture} 
			\begin{axis}[
				title={{$p=30$}},
				xmin=-0.25, 
				xmax=10.75,
				ymin=10, 
				ymax=12000,
				ymode=log,
				xlabel={\% Primal Gap}, 
				ylabel={Time (s)},
				scatter/classes={
					t={scale=1.95,mark=triangle*, blue}, 
					d00={scale=1.95,mark=halfdiamond*, mark color=white, fill=white}, 
					d01={scale=1.95,mark=halfdiamond*, mark color=white, fill=black}, 
					d10={scale=1.95,mark=halfdiamond*, mark color=black, fill=white}, 
					d11={scale=1.95,mark=halfdiamond*, mark color=black, fill=black}, 
					c00={scale=1.95,mark=halfcircle*, fill=white, mark color=white}, 
					c01={scale=1.95,mark=halfcircle*, fill=white, mark color=red}, 
					c10={scale=1.95,mark=halfcircle*, fill=red, mark color=white}, 
					c11={scale=1.95,mark=halfcircle*, fill=red, mark color=red}, 
					s00={scale=1.95,mark=halfsquare*, mark color=white, fill=white},
					s10={scale=1.95,mark=halfsquare*, mark color=green, fill=white}
				},
				] 
				
				\addplot[
				scatter, 
				only marks,
				scatter src=explicit symbolic,
				nodes near coords*={\annotvalue},
				node near coord style={ anchor=\anchorvalue, font=\scriptsize, color=\labelcolor},
				visualization depends on={value \thisrow{annotation} \as \annotvalue },
				visualization depends on={  value \thisrow{anchor} \as \anchorvalue },
				]
				table[meta=label] {
					x       y            label  annotation  anchor
					0.182	2393.24	t	1	south
					0.198	3451.82	t	2	south
					1.747	998.36	d10	3 south
					1.713	163.92	d10	4	south
					2.115	1070.00	d11	5	south
					2.093	177.24	d11	6	south
					6.276	1008.74	d00	7	south
					6.168	162.42	d00	8 south
					6.532	997.10	d01	9	south
					6.447	176.74	d01	10	south
					4.287	1019.46	c10	11	south
					4.215	162.06	c10	12	south
					4.809	988.56	c11	13	south
					4.764	174.38	c11	14	south
					7.292	973.42	c00	15	south
					7.198	162.84	c00	16	south
					7.448	1032.54	c01	17	south
					7.453	173.06	c01	18	south
					6.700	1054.72	s10	19	south
					6.748	160.86	s10	20	south
					8.746	1026.62	s00	21	south
					8.683	165.32	s00	22	south
				};

				\addplot[magenta, dashed] coordinates
				{(-0.02,700) (9.7,700) (9.7,2900) (-0.02,2900) (-0.02,700)};
				
				
				\addplot[cyan] coordinates
				{(1.2,120) (9.7,120) (9.7,360) (1.2,360) (1.2,120)};
				
				
				\addplot[cyan] coordinates
				{(-0.02,3000) (0.45,3000) (0.45,4500) (-0.02,4500) (-0.02,3000)};
			\end{axis}
			
		\end{tikzpicture}
	\end{subfigure}
	\begin{subfigure}{.5\textwidth}
		\begin{tikzpicture}
			
			\begin{axis}[
				title={{$p=40$}},
				xmin=-0.25, 
				xmax=10.75,
				ymin=10, 
				ymax=12000,
				ymode=log,
				xlabel={\% Primal Gap}, 
				scatter/classes={
					t={scale=1.95,mark=triangle*, blue}, 
					d00={scale=1.95,mark=halfdiamond*, mark color=white, fill=white}, 
					d01={scale=1.95,mark=halfdiamond*, mark color=white, fill=black}, 
					d10={scale=1.95,mark=halfdiamond*, mark color=black, fill=white}, 
					d11={scale=1.95,mark=halfdiamond*, mark color=black, fill=black}, 
					c00={scale=1.95,mark=halfcircle*, fill=white, mark color=white}, 
					c01={scale=1.95,mark=halfcircle*, fill=white, mark color=red}, 
					c10={scale=1.95,mark=halfcircle*, fill=red, mark color=white}, 
					c11={scale=1.95,mark=halfcircle*, fill=red, mark color=red}, 
					s00={scale=1.95,mark=halfsquare*, mark color=white, fill=white},
					s10={scale=1.95,mark=halfsquare*, mark color=green, fill=white}
				},
				] 
				
				\addplot[
				scatter, 
				only marks,
				scatter src=explicit symbolic,
				nodes near coords*={\annotvalue},
				node near coord style={ anchor=\anchorvalue, font=\scriptsize, color=\labelcolor},
				visualization depends on={value \thisrow{annotation} \as \annotvalue },
				visualization depends on={  value \thisrow{anchor} \as \anchorvalue },
				]
				table[meta=label] {
					x       y            label  annotation  anchor
					0.094	6242.24	t	1	south
					0.079	9029.84	t	2	south
					1.437	2418.64	d10	3	south
					1.361	395.80	d10	4 south
					1.641	2557.04	d11	5	south
					1.535	436.88	d11	6	south
					6.547	2383.08	d00	7	south
					6.562	382.00	d00	8	south
					6.749	2397.84	d01	9	south
					6.825	429.88	d01	10	south
					4.253	2385.68	c10	11	south
					4.300	382.80	c10	12	south
					4.788	2357.36	c11	13	south
					4.833	418.24	c11	14	south
					7.684	2344.76	c00	15	south
					7.655	386.36	c00	16	south
					7.937	2419.76	c01	17	south
					7.891	414.24	c01	18	south
					6.400	2505.04	s10	19	south
					6.611	379.96	s10	20	south
					8.540	2431.72	s00	21	south
					8.723	391.84	s00	22	south
				};
				
				\addplot[magenta, dashed] coordinates
				{(-0.12,1800) (9.7,1800) (9.7,7500) (-0.12,7500) (-0.12,1800)};
				
				
				\addplot[cyan] coordinates
				{(1.0,250) (9.7,250) (9.7,750) (1.0,750) (1.0,250)};
				
				
				\addplot[cyan] coordinates
				{(-0.12,8000) (0.3,8000) (0.3,11500) (-0.12,11500) (-0.12,8000)};
			\end{axis}
			
		\end{tikzpicture}
	\end{subfigure}
\end{figure}

In Figure~\ref{fig:deltaGap}, we report the  \% Primal Gap vs. Time, where each metric is averaged over instances. We have several interesting observations from these experiments:
\begin{itemize}
	\item In terms of the \% Primal Gap, the most successful heuristics are Version 1 and 2, in which we only apply an edge filter. Since these heuristics only fix a small number of variables, it is conceivable that they provide the highest quality feasible solutions. However, this comes with the price of solving a difficult MILP at the end of Algorithm~\ref{alg:feasible-solution}. In terms of Time, the MTZ formulation is faster for $p=10$ and the Flow formulation is faster for $p=20,30,40$. This is consistent with the results of the experiments run with the MILP models. We also observe that these two versions are slower than directly using the MILP models for $p=40$, therefore, they are not advantageous in the case of more difficult instances.
	\item 
	For some instances, Version 1 or Version 2 are able to find higher quality feasible solutions that the best MILP solution found (i.e., the \% Primal Gap is negative). In particular, there are 4/75, 5/50 and 6/25 instances for $p=20$, $p=30$ and $p=40$, respectively, for which the heuristic solutions are of higher quality. Most of these instances come from  the Bench family and a few are from the pMed family. 
	This underscores the value of heuristic approaches in the case of more difficult instances.
	\item Overall, the versions that do not apply the edge filter are dominated by those that do apply this filter in terms of \% Primal Gap. Therefore, it is crucial to apply the edge filter to obtain high quality feasible solutions.
	\item Apart from Versions 1 and 2, the most successful heuristics in terms of the \% Primal Gap are Versions 3 and 4, in which the node filter \texttt{LV} is applied in addition to the edge filter. These versions are closely followed Versions 5 and 6 in which additional reduction step is taken.
	\item 
	Versions that use the node filter \texttt{LV-g} seem to be outperformed by versions that use the node filter \texttt{LV} in terms of the \% Primal Gap.
	\item Versions 21 and 22 are consistently the worst performing heuristics. We observe that the versions utilizing node filter \texttt{Ch} are not very successful in general. 
	\item In terms of Time, except for Versions 1 and 2, all  the remaining versions are similar for the MTZ and Flow formulations, where the former is much faster and is the preferable formulation.
	\item 
	In general, the reduction step applied does not seem to decrease the computational effort as initially intended and worsens the \% Primal Gap. Therefore, this additional step is not advisable.
	\item As expected, the CPU time of each version increases with $p$. However, it is interesting to observe that the \% Primal Gap values tend to decrease with $p$. 
\end{itemize}
To summarize, time consuming heuristics Version 1 (for $p=20,30,40$) and Version 2 (for $p=10$) are the best  if solution quality is prioritized. On the other hand, Version 4 is the best heuristic that balances quality and computational effort successfully. This is the reason it is used in the experiments conducted for the large-scale instances  reported in Section~\ref{sec:largeScale}.

%
%
%
%
%
%

In Figure~\ref{fig:lpGap}, we report the  \% Dual Gap vs. Time, where each metric is averaged over instances. Most of our observations from Figure~\ref{fig:deltaGap} carry over here as well.

\begin{figure}
	\caption{\% Dual Gap vs. Time. The markers in the rectangle with solid (resp. dashed) border use  the \texttt{MTZ} (resp. \texttt{Flow}) formulation.}
	\label{fig:lpGap}
	
	\begin{subfigure}{.5\textwidth}
		\begin{tikzpicture} 
			\begin{axis}[
				title={{$p=10$}},
				xmin=-0.25, 
				xmax=11.25,
				ymin=10, 
				ymax=12000,
				ymode=log,
				ylabel={Time (s)},
				scatter/classes={
					t={scale=1.95,mark=triangle*, blue}, 
					d00={scale=1.95,mark=halfdiamond*, mark color=white, fill=white}, 
					d01={scale=1.95,mark=halfdiamond*, mark color=white, fill=black}, 
					d10={scale=1.95,mark=halfdiamond*, mark color=black, fill=white}, 
					d11={scale=1.95,mark=halfdiamond*, mark color=black, fill=black}, 
					c00={scale=1.95,mark=halfcircle*, fill=white, mark color=white}, 
					c01={scale=1.95,mark=halfcircle*, fill=white, mark color=red}, 
					c10={scale=1.95,mark=halfcircle*, fill=red, mark color=white}, 
					c11={scale=1.95,mark=halfcircle*, fill=red, mark color=red}, 
					s00={scale=1.95,mark=halfsquare*, mark color=white, fill=white},
					s10={scale=1.95,mark=halfsquare*, mark color=green, fill=white}
				},
				] 
				\addplot[
				scatter, 
				only marks,
				scatter src=explicit symbolic,
				nodes near coords*={\annotvalue},
				node near coord style={ anchor=\anchorvalue, font=\scriptsize, color=\labelcolor},
				visualization depends on={value \thisrow{annotation} \as \annotvalue },
				visualization depends on={  value \thisrow{anchor} \as \anchorvalue },
				]
				table[meta=label] {
					x       y            label  annotation  anchor
					2.934531264	206.50	t	1	south
					2.936267582	62.15	t	2	south
					4.591722719	174.55	d10	3	south
					4.5327235	32.17	d10	4	south
					4.850243019	177.33	d11	5	south
					4.80137766	33.22	d11	6	south
					6.152821958	181.68	d00	7	south
					6.031050553	32.41	d00	8	south
					6.436128777	171.36	d01	9	south
					6.321359363	35.40	d01	10	south
					5.738704271	181.94	c10	11	south
					5.764742422	32.59	c10	12	south
					6.20828872	168.19	c11	13	south
					6.247168087	33.33	c11	14	south
					7.145011005	176.09	c00	15	south
					7.077230926	32.79	c00	16	south
					7.363778761	190.72	c01	17	south
					7.288535295	33.93	c01	18	south
					10.23546838	187.32	s10	19	south
					10.10023417	32.43	s10	20	south
					10.90757161	185.97	s00	21	south
					10.84535998	32.90	s00	22	south
				};
				
				\addplot[magenta, dashed] coordinates
				{(2.5,100) (11.15,100) (11.15,300) (2.5,300) (2.5,100)};
				
				\addplot[cyan] coordinates
				{(2.5,25) (11.15,25) (11.15,80) (2.5,80) (2.5,25)};
			\end{axis}
			
		\end{tikzpicture}
	\end{subfigure}
	\begin{subfigure}{.5\textwidth}
		\begin{tikzpicture}
			
			\begin{axis}[
				title={{$p=20$}},
				xmin=-0.25, 
				xmax=11.25,
				ymin=10, 
				ymax=12000,
				ymode=log,
				scatter/classes={
					t={scale=1.95,mark=triangle*, blue}, 
					d00={scale=1.95,mark=halfdiamond*, mark color=white, fill=white}, 
					d01={scale=1.95,mark=halfdiamond*, mark color=white, fill=black}, 
					d10={scale=1.95,mark=halfdiamond*, mark color=black, fill=white}, 
					d11={scale=1.95,mark=halfdiamond*, mark color=black, fill=black}, 
					c00={scale=1.95,mark=halfcircle*, fill=white, mark color=white}, 
					c01={scale=1.95,mark=halfcircle*, fill=white, mark color=red}, 
					c10={scale=1.95,mark=halfcircle*, fill=red, mark color=white}, 
					c11={scale=1.95,mark=halfcircle*, fill=red, mark color=red}, 
					s00={scale=1.95,mark=halfsquare*, mark color=white, fill=white},
					s10={scale=1.95,mark=halfsquare*, mark color=green, fill=white}
				},
				] 
				
				\addplot[
				scatter, 
				only marks,
				scatter src=explicit symbolic,
				nodes near coords*={\annotvalue},
				node near coord style={ anchor=\anchorvalue, font=\scriptsize, color=\labelcolor},
				visualization depends on={value \thisrow{annotation} \as \annotvalue },
				visualization depends on={  value \thisrow{anchor} \as \anchorvalue },
				]
				table[meta=label] {
					x       y            label  annotation  anchor
					1.693	877.55	t	1	south
					1.696	1310.03	t	2	south
					3.279	470.80	d10	3	south
					3.379	80.05	d10	4	south
					3.664	499.92	d11	5	south
					3.762	84.55	d11	6	south
					6.525	475.05	d00	7	south
					6.613	80.49	d00	8	south
					6.992	462.12	d01	9	south
					7.044	85.72	d01	10	south
					4.969	467.23	c10	11	south
					5.090	79.79	c10	12	south
					5.550	468.48	c11	13	south
					5.682	84.68	c11	14	south
					7.440	463.77	c00	15	south
					7.409	80.16	c00	16	south
					7.723	484.57	c01	17	south
					7.766	83.68	c01	18	south
					7.932	510.55	s10	19	south
					8.074	79.68	s10	20	south
					9.354	466.17	s00	21	south
					9.303	81.51	s00	22	south
				};

				\addplot[magenta, dashed] coordinates
				{(1.45,300) (9.7,300) (9.7,1100) (1.45,1100) (1.45,300)};
				
				
				\addplot[cyan] coordinates
				{(3,50) (9.7,50) (9.7,150) (3,150) (3,50)};
				
				
				\addplot[cyan] coordinates
				{(1.45,1150) (1.95,1150) (1.95,1600) (1.45,1600) (1.45,1150)};
				
			\end{axis}
			
		\end{tikzpicture}
	\end{subfigure}
	
	\begin{subfigure}{.5\textwidth}
		\begin{tikzpicture} 
			\begin{axis}[
				title={{$p=30$}},
				xmin=-0.25, 
				xmax=11.25,
				ymin=10, 
				ymax=12000,
				ymode=log,
				xlabel={\% Dual Gap}, 
				ylabel={Time (s)},
				scatter/classes={
					t={scale=1.95,mark=triangle*, blue}, 
					d00={scale=1.95,mark=halfdiamond*, mark color=white, fill=white}, 
					d01={scale=1.95,mark=halfdiamond*, mark color=white, fill=black}, 
					d10={scale=1.95,mark=halfdiamond*, mark color=black, fill=white}, 
					d11={scale=1.95,mark=halfdiamond*, mark color=black, fill=black}, 
					c00={scale=1.95,mark=halfcircle*, fill=white, mark color=white}, 
					c01={scale=1.95,mark=halfcircle*, fill=white, mark color=red}, 
					c10={scale=1.95,mark=halfcircle*, fill=red, mark color=white}, 
					c11={scale=1.95,mark=halfcircle*, fill=red, mark color=red}, 
					s00={scale=1.95,mark=halfsquare*, mark color=white, fill=white},
					s10={scale=1.95,mark=halfsquare*, mark color=green, fill=white}
				},
				] 
				
				\addplot[
				scatter, 
				only marks,
				scatter src=explicit symbolic,
				nodes near coords*={\annotvalue},
				node near coord style={ anchor=\anchorvalue, font=\scriptsize, color=\labelcolor},
				visualization depends on={value \thisrow{annotation} \as \annotvalue },
				visualization depends on={  value \thisrow{anchor} \as \anchorvalue },
				]
				table[meta=label] {
					x       y            label  annotation  anchor
					1.106	2393.24	t	1	south
					1.121	3451.82	t	2	south
					2.616	998.36	d10	3	south
					2.586	163.92	d10	4	south
					2.961	1070.00	d11	5	south
					2.941	177.24	d11	6	south
					6.775	1008.74	d00	7	south
					6.680	162.42	d00	8	south
					6.990	997.10	d01	9	south
					6.919	176.74	d01	10	south
					4.982	1019.46	c10	11	south
					4.918	162.06	c10	12	south
					5.452	988.56	c11	13	south
					5.410	174.38	c11	14	south
					7.655	973.42	c00	15	south
					7.574	162.84	c00	16	south
					7.785	1032.54	c01	17	south
					7.789	173.06	c01	18	south
					7.119	1054.72	s10	19	south
					7.169	160.86	s10	20	south
					8.882	1026.62	s00	21	south
					8.831	165.32	s00	22	south
				};

				\addplot[magenta, dashed] coordinates
				{(0.9,700) (9.7,700) (9.7,2900) (0.9,2900) (0.9,700)};
				
				
				\addplot[cyan] coordinates
				{(2,120) (9.7,120) (9.7,360) (2,360) (2,120)};
				
				
				\addplot[cyan] coordinates
				{(0.9,3000) (1.4,3000) (1.4,4500) (0.9,4500) (0.9,3000)};
			\end{axis}
			
		\end{tikzpicture}
	\end{subfigure}
	\begin{subfigure}{.5\textwidth}
		\begin{tikzpicture}
			
			\begin{axis}[
				title={{$p=40$}},
				xmin=-0.25, 
				xmax=11.25,
				ymin=10, 
				ymax=12000,
				ymode=log,
				xlabel={\% Dual Gap}, 
				scatter/classes={
					t={scale=1.95,mark=triangle*, blue}, 
					d00={scale=1.95,mark=halfdiamond*, mark color=white, fill=white}, 
					d01={scale=1.95,mark=halfdiamond*, mark color=white, fill=black}, 
					d10={scale=1.95,mark=halfdiamond*, mark color=black, fill=white}, 
					d11={scale=1.95,mark=halfdiamond*, mark color=black, fill=black}, 
					c00={scale=1.95,mark=halfcircle*, fill=white, mark color=white}, 
					c01={scale=1.95,mark=halfcircle*, fill=white, mark color=red}, 
					c10={scale=1.95,mark=halfcircle*, fill=red, mark color=white}, 
					c11={scale=1.95,mark=halfcircle*, fill=red, mark color=red}, 
					s00={scale=1.95,mark=halfsquare*, mark color=white, fill=white},
					s10={scale=1.95,mark=halfsquare*, mark color=green, fill=white}
				},
				] 
				
				\addplot[
				scatter, 
				only marks,
				scatter src=explicit symbolic,
				nodes near coords*={\annotvalue},
				node near coord style={ anchor=\anchorvalue, font=\scriptsize, color=\labelcolor},
				visualization depends on={value \thisrow{annotation} \as \annotvalue },
				visualization depends on={  value \thisrow{anchor} \as \anchorvalue },
				]
				table[meta=label] {
					x       y            label  annotation  anchor
					0.842868472	6242.24	t	1	south
					0.828667295	9029.84	t	2	south
					2.14607879	2418.64	d10	3	south
					2.07496187	395.80	d10	4	south
					2.340832535	2557.04	d11	5	south
					2.24003186	436.88	d11	6	south
					6.845114977	2383.08	d00	7	south
					6.858701728	382.00	d00	8	south
					7.01947331	2397.84	d01	9	south
					7.086812423	429.88	d01	10	south
					4.791949203	2385.68	c10	11	south
					4.835124213	382.80	c10	12	south
					5.272409177	2357.36	c11	13	south
					5.313553088	418.24	c11	14	south
					7.82604771	2344.76	c00	15	south
					7.801226172	386.36	c00	16	south
					8.036389838	2419.76	c01	17	south
					8.001702244	414.24	c01	18	south
					6.697293617	2505.04	s10	19	south
					6.886353201	379.96	s10	20	south
					8.55234764	2431.72	s00	21	south
					8.706436429	391.84	s00	22	south
				};
				
				\addplot[magenta, dashed] coordinates
				{(0.55,1800) (9.20,1800) (9.20,7500) (0.55,7500) (0.55,1800)};
				
				
				\addplot[cyan] coordinates
				{(1.855,250) (9.20,250) (9.20,750) (1.855,750) (1.855,250)};
				
				
				\addplot[cyan] coordinates
				{(.55,8000) (1.05,8000) (1.05,11500) (.55,11500) (.55,8000)};
			\end{axis}
			
		\end{tikzpicture}
	\end{subfigure}
\end{figure}
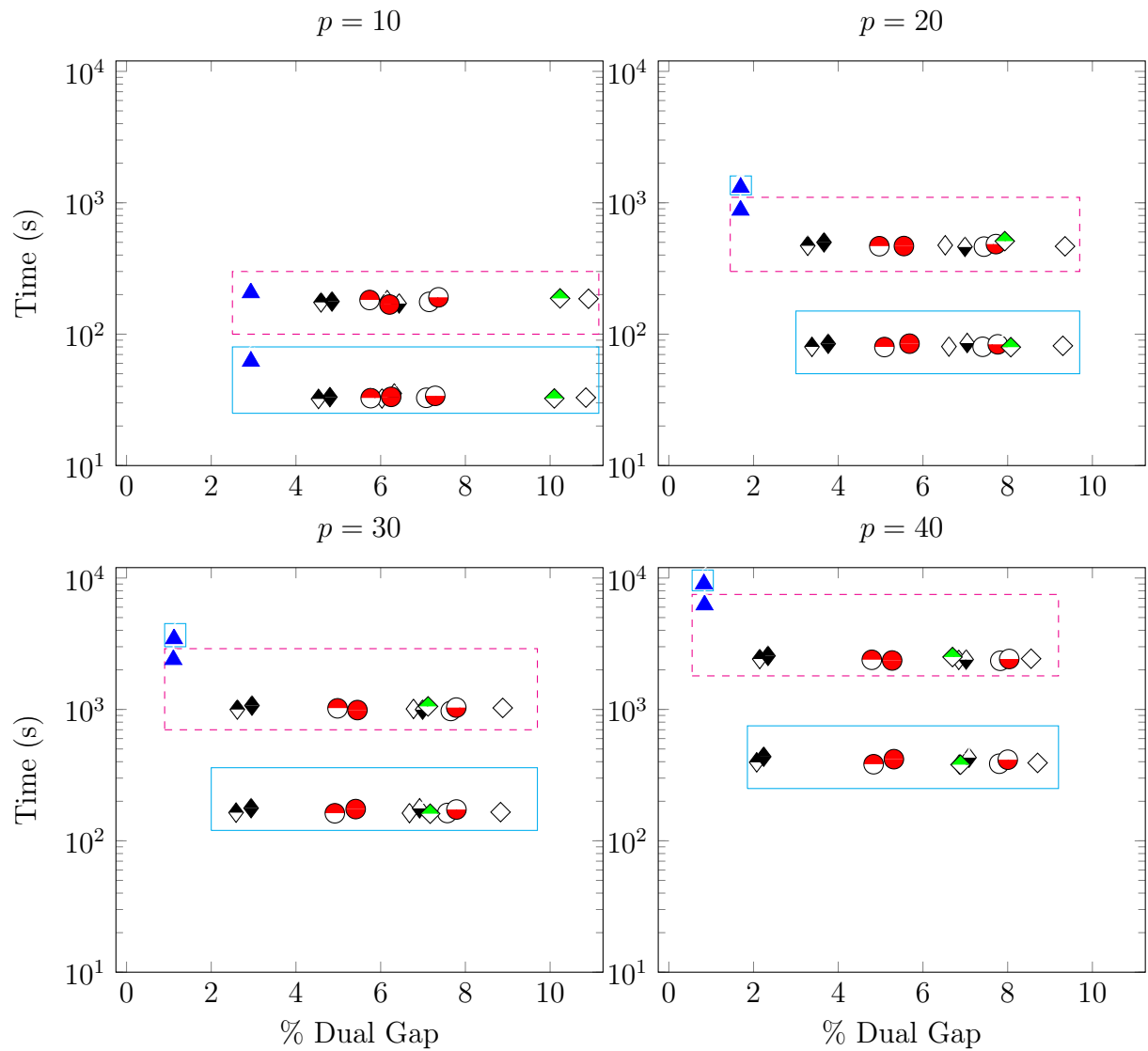

\subsubsection{Large-Scale Instances}
\label{sec:largeScale}

To test the capability of Version 4 to the fullest extent, five random samples from each graph type with 900 nodes are generated and we set $p=90$. 
Since the pMed instances are taken from the literature and only three instances have 900 nodes, two additional instances are produced by changing data and deployment costs. Since it is impractical to solve MILP model for such large instances, we are unable to report the \% Primal Gap. Therefore, we only report \% Dual Gap to show the effectiveness of our approach. 

We have the following observations as a result of the experiments:

\begin{itemize}
	\item Table~\ref{tab:large} demonstrates that the duality gap is at most 3\% for these instances. Considering the large-scale nature of these instances, we think that these results are satisfactory.
	\item In this table, although ER and FF have similarly {high} densities, their Time is dramatically different. This is an interesting observation since solving the MILP at Phase 4  is relatively easy considering our additional experiments on pMed instances with $|N|=100$ and $p=90$, which are solved within 68.6 seconds on average. However, two of the FF instances hit the  two-hour time limit at Phase 4. Therefore, we reach the  conclusion that graph structure plays a significant role in accounting for Time differences {in high densities}.
	\item
	In low densities, BA, bench and pmed instances yield similar results with respect to Time. So, the effect of graph structure seems to diminish for sparse graphs unlike dense ones.
	
\end{itemize}

\begin{table}[H]
	\centering
	\footnotesize
	\caption{Version 4 results for large-scale instances with $|N|=900$  and $p=90$.}
	\label{tab:large}
	\begin{tabular}{|ccc|cr|}
		\hline
		{instance} &  $|E|$ & {density} & {\% Dual Gap} & {Time} \\ \hline
		er-900              & 121425.20  & 0.30             & 1.03             & 691.00        \\
		ba-900             & 16731.00   & 0.04             & 1.39             & 3479.40       \\
		bench-900            & 20035.60   & 0.05             & 1.33             & 3661.40       \\
		ff-900               & 124364.80  & 0.31             & 0.59             & 3453.40       \\
		pmed-900             & 16200.00   & 0.04             & 2.88             & 3699.60      \\ \hline
	\end{tabular}
\end{table}


\section{Conclusions}
\label{sec:conclusions}

In this work, we introduced a new NP-Hard problem inspired by routing of information from sensors to sinks and the share of information between sinks. On top of that, we also incorporate the uncapacitated facility location and p-median problem aspects because these sensors have communication data and we are required to deploy $p$ sinks that are able to serve all nodes. 
\textcolor{black}{Furthermore, an LP rounding-based four-phase matheuristic is devised to address the challenges posed by large-scale instances, where obtaining an optimal solution is often computationally prohibitive.} So as to generalize our findings, we created a test bed based on four well known graph types and one graph type taken from the literature while system parameters like deployment cost, data, edge cost and connection costs are generated according to our system of interest which is the system with high deployment cost.

We performed extensive computational experiments. According to the exact MILP results, the CPU time increases significantly with the number of nodes whereas there is no clear relation  between the parameter $p$ and the CPU Time. However, as $p$ increases, the optimality gap decreases suggesting that finding a larger induced spanning tree is easier than finding a smaller one. When it comes to the graph structure effect, experiments on the pMed instances result in the worst CPU time and \% Gap.
One might argue that the pMed is the sparsest, and thus, its sparsity explains the observed results. However, this is not the case. In fact, Bench is the second sparsest, yet it is easier to solve compared to the FF, which has a significantly higher density. Furthermore, when examining dense graphs, the experimental results for FF and ER reveal substantial differences in KPIs. These observations suggest that graph structure, rather than density alone, plays a critical role in explaining the variability in KPIs.
According to model performances, MTZ surpasses Flow on \% Gap slightly while giving up \%13 - \% 18 CPU Time. So, we recommend practitioners Flow model if they can tolerate the slight difference in solution quality.


We analyze three cost components under MTZ. The results suggest that as $p$ increases, the access cost decreases because each node is connected to its closest sink at a smaller cost. On the flip side, connection and deployment costs increase simultaneously with $p$. We observe larger access cost and thus larger total cost in sparse Bench and pMed instances. 

Our matheuristic approach has many versions depending on parameter selection in different algorithm phases. We compare the performance of these versions among themselves with regard to \% Primal Gap and Time in a bi-objective manner. The results demonstrate that the versions which do not apply \texttt{edgeFilter} and \texttt{reduction} yield the lowest \% Primal gap, and hence high quality feasible solutions. They even generate better outcomes than the best MILP solution in some Bench and pMed instances. When \texttt{nodeFilter} comes into play, applying \citet{lin1992approximations}'s method for \texttt{nodeFilter} yields significantly faster but slightly worse solution quality than two high quality versions without \texttt{nodeFilter}. Similar conclusions are drawn from \% Dual Gap results; thus, the latter approach, balancing Time and solution quality, is selected for large-scale experiment for sensitivity analysis. 

The results on large-scale instances show that the \% Dual Gap is at most 3\% on each graph type. Next, among dense graphs, there is a significant difference in computation time between the two graph types, which can be attributed to variations in their structural properties. However, this effect diminishes in graphs with similarly low densities, where the impact of structure becomes less pronounced. 

There are promising future research directions. From a methodological point of view, one can develop a Benders Decomposition algorithm for better computational efficiency. Secondly, this problem can be differentiated by relaxing the spanning tree assumption between sinks. In the new setting, sinks could be connected via Steiner Tree as in the Connected Facility Location Problem. Our developed matheuristic could be used as a mean to compare the total cost between two different settings and can draw insights for the practitioners or researchers working with sensor networks.

\section*{Acknowledgment}
Murat Elhüseyni and Miklós Krész have been supported by the Slovenian Research and Innovation Agency (ARIS) through grant J2-2504. Miklós Krész is grateful for the support of the ARIS grants N2-0434, N2-0486, J1-70046, BI-HU/26-27-006 and BI-HU/26-27-007. He has been also supported by the research program CogniCom (0013103) at the University of Primorska. 

\section*{Data Statement}


The data and codes are available at 
\url{https://github.com/muratelhuseyni/SensorNetworkOptimization/tree/main/
mathematical\%20programming}




  \bibliographystyle{elsarticle-harv} 
  \bibliography{lit}






\end{document}